\documentclass[11pt,reqno]{amsart}
\usepackage{color}
\usepackage{graphicx,array}
\usepackage{bbm}
\usepackage{amsmath,amsthm}

\newcommand{\ul}{\underline}

\newtheorem{theorem}{Theorem}[section]
\newtheorem{corollary}[theorem]{Corollary}

\newtheorem{lemma}[theorem]{Lemma}

\newtheorem{example}[theorem]{Example}

\theoremstyle{remark}

\theoremstyle{definition}
\newtheorem{definition}[theorem]{Definition}

\newcommand{\ds}{\displaystyle}

\DeclareMathOperator{\sss}{{\mathfrak{S}}}
\DeclareMathOperator{\sn}{{\sss}_{\mathit{n}}}
\DeclareMathOperator{\trace}{\mathrm{trace}}

\DeclareMathOperator{\Des}{\mathsf{Des}}
\DeclareMathOperator{\Dess}{\mathsf{Des}^{*}}
\DeclareMathOperator{\des}{\mathsf{des}}
\DeclareMathOperator{\dess}{\mathsf{des}^{*}}
\DeclareMathOperator{\Equ}{\mathsf{Equ}}
\DeclareMathOperator{\Equs}{\mathsf{Equ}^{*}}
\DeclareMathOperator{\equ}{\mathsf{equ}}
\DeclareMathOperator{\equs}{\mathsf{equ}^{*}}
\DeclareMathOperator{\Asc}{\mathsf{Asc}}
\DeclareMathOperator{\Ascs}{\mathsf{Asc}^{*}}

\DeclareMathOperator{\ascs}{\mathsf{asc}^{*}}
\DeclareMathOperator{\DEA}{\mathsf{DEA}}
\DeclareMathOperator{\DEAS}{\mathsf{DEA}^{*}}
\DeclareMathOperator{\rel}{\mathsf{rel}}
\DeclareMathOperator{\Minimal}{\mathrm{Minimal}}
\DeclareMathOperator{\minimal}{\mathrm{minimal}}
\DeclareMathOperator{\NoEquals}{\mathrm{NoEquals}}
\DeclareMathOperator{\PegSet}{\mathrm{PegSet}}
\DeclareMathOperator{\EdgeSet}{\mathrm{EdgeSet}}
\DeclareMathOperator{\PegpairsSet}{\mathrm{PegpairsSet}}

\newcommand{\Keys}[3]{\mathsf{Keys}^{#3}(#1,#2)}
\newcommand{\stirling}[2]{\left\{ #1 \atop #2 \right\}}

\newcommand{\Recon}[2]{{\mathsf{Recon}}(#1,#2)}
\newcommand{\MMmatrix}[3]{{\mathfrak{M}}^{#1}_{#3} {#2}}
\newcommand{\RRmatrix}[2]{{\mathfrak{R}}^{#1}_{#2}}
\newcommand\Colours[2]{{\mathsf{Colors}}({#1,#2})}
\newcommand{\Weuler}{\mathsf{WordEuler}}
\newcommand{\weuler}{\mathsf{wordeuler}}
\newcommand{\Weulers}{\mathsf{WordEuler}^{*}}
\newcommand{\weulers}{\mathsf{wordeuler}^{*}}
\newcommand{\Represent}{\mathsf{Represent}}
\newcommand{\nww}{\mathsf{nww}}
\newcommand{\nwwnip}{\mathsf{nwwnip}}
\newcommand{\npww}{\mathsf{npww}}
\newcommand{\NPWW}{\mathsf{NWWNIP}}
\newcommand{\wdm}{\mathsf{WDM}}
\newcommand{\peg}{p}
\newcommand{\Pegs}{\mathrm{Pegs}}
\newcommand{\WW}{W} 
\newcommand{\AllPosets}{\mathsf{AllPosets}}
\newcommand{\multip}[2]{\mathsf{multip}_{#1}(#2)}
\newcommand{\NWW}{\mathsf{NWW}}
\newcommand{\N}{\mathbb{N}}

\title{Web worlds, web-colouring matrices, and web-mixing matrices}
\author[M. Dukes]{Mark Dukes} 
\author[E. Gardi]{Einan Gardi} 
\author[E. Steingr\'imsson]{Einar Steingr\'imsson} 
\author[C. D. White]{Chris D. White} 

\address{MD \& ES: Department of Computer and Information Sciences,
  University of Strathclyde, Glasgow, G1 1XH, UK} \address{EG: School
  of Physics and Astronomy, University of Edinburgh, Edinburgh, EH9
  3JZ, UK} \address{CDW: School of Physics and Astronomy, University
  of Glasgow, Glasgow, G12 8QQ, UK}

\thanks{MD \& ES: The work presented here was supported by grant no.\
  090038013 from the Icelandic Research Fund.}


\begin{document}
\thispagestyle{empty}
\maketitle

\begin{abstract}
We introduce a new combinatorial object called a \emph{web world} that consists of a set of \emph{web diagrams}.  The diagrams of a web world are generalizations of graphs, and each is built on the same underlying graph.  Instead of ordinary vertices the diagrams have \emph{pegs}, and edges incident to a peg have different heights on the peg.  The web world of a web diagram is the set of all web diagrams that result from permuting the order in which endpoints of edges appear on a peg. The motivation comes from particle physics, where web diagrams arise as particular types of Feynman diagrams describing scattering amplitudes in non-Abelian gauge (Yang-Mills) theories.   To each web world we associate two matrices called the \emph{web-colouring matrix} and \emph{web-mixing matrix}.  The entries of these matrices are indexed by ordered pairs of web diagrams $(D_1,D_2)$, and are computed from those colourings of the edges of $D_1$ that yield $D_2$ under a transformation determined by each colouring.

We show that colourings of a web diagram 
(whose constituent indecomposable diagrams are all unique) 
that lead to a reconstruction of the diagram are equivalent to 
order-preserving mappings of certain partially ordered sets (posets) 
that may be constructed from the web diagrams.
For web worlds whose web graphs have all edge labels equal to 1, 
the diagonal entries of web-mixing and web-colouring matrices 
are obtained by summing certain
polynomials determined by the descents in permutations in the
Jordan-H\"older set of all linear extensions of the associated
poset.
We derive tri-variate generating generating functions for the number of web worlds according to three statistics and enumerate the number of different web diagrams in a web world.
Three special web worlds are examined in great detail, and the traces of the 
web-mixing matrices calculated in each case.
\end{abstract}

\section{Introduction}
In this paper we study combinatorial objects called {\it{web worlds}}. 
A web world consists of a set of diagrams that we call {\it{web diagrams}}.
The motivation for introducing these comes from
particle physics, where web diagrams arise as particular types of
Feynman diagrams describing scattering amplitudes of quantum fields in
non-Abelian gauge (Yang-Mills) theories.  These structures, which
do not seem to have been studied previously from the purely
combinatorial point of view, may be thought of as generalisations of
simple graphs where each edge has a height associated to each of its
endpoints.  To avoid confusion with proper graphs, the vertices of our
web diagrams will be referred to as {\it{pegs}}.  The different edges
that connect onto a given peg are strictly ordered by their heights as
illustrated in the example in Figure~\ref{first:example}.  The
\emph{web world} of a web diagram is the set of all web diagrams that
result from permuting the order in which endpoints of edges appear on
a peg, that is, the heights of the edges incident with that peg.  To
each web world we associate two matrices, $\MMmatrix{}{(x)}{}$ and
$\RRmatrix{}{}$, which are called the {\it{web-colouring matrix}} and
{\it{web-mixing matrix}}, respectively.  The entries of these matrices
are indexed by ordered pairs of web diagrams, and are computed from
those colourings of the edges of one of the two web diagrams that
yield the other one under a certain transformation 
(Definition~\ref{twonine}) determined by each colouring.

Web diagrams, web worlds and their web-mixing matrices play an important
role in the study of scattering amplitudes in quantum chromodynamics
(QCD)~\cite{WIM,MSS,GPO,OTR}. While not essential for the
combinatorial aspects dealt with here, we briefly describe the physics
context in which they arise.
\begin{itemize}
\item{} In general, in perturbative quantum field theory, Feynman
  diagrams provide means of computing scattering amplitudes: each
  diagram corresponds to an algebraic expression depending on the
  momenta and other quantum numbers of the external
  particles. Specifically in QCD these diagrams describe the
  collisions of energetic quarks and gluons (generically referred to
  as partons, the constituents of protons and neutrons); these carry a
  matrix-valued Yang-Mills charge belonging respectively to the
  fundamental or adjoint representations of the ${\rm SU}(N)$
  group. These matrices do not commute, hence the name non-Abelian
  gauge theory.
\item{} The perturbation expansion of a scattering amplitude, formally
  a power series in the coupling constant (the strength of the
  interaction) amounts to a loop expansion: the leading term is the
  sum of all diagrams having a tree topology connecting all external
  particles, while an $n$-th order term in the loop expansion
  corresponds to the sum of all diagrams involving $n$ loops, which is
  obtained by dressing a tree diagram by $n$ additional gluon
  exchanges. In a non-Abelian theory gluons may connect to each other
  via 3 and 4 gluon vertices.
\item{} A salient feature of scattering amplitudes is that starting at
  one loop they involve long-distance (``soft'') singularities.  In
  order to study the structure of these singularities one may consider
  an \emph{Eikonal amplitude}, which is a simplified version of the
  QCD scattering amplitude that fully captures its singularities.
  This simplification is ultimately a consequence of the fact that all
  singularities arise due to ``soft'' (low energy) fields, and the
  quantum-mechanical incoherence of ``soft'' and ``hard'' (high
  energy) fields.  Eikonal amplitudes are formed by representing each
  external parton by an Eikonal line (or ``Wilson line''), a single
  semi-infinite line extending from the origin to infinity in the
  direction fixed by the momentum of the external parton, and which
  carries the same matrix-valued Yang-Mills charge.
\item{} Eikonal amplitudes \emph{exponentiate}: they can be written as
  an exponential where the exponent takes a particularly simple
  form. Web diagrams arise as a direct Feynman--diagrammatic
  description of this exponent~\cite{WIM,MSS}, namely, they define the
  coefficients in the loop expansion of the exponent of the Eikonal
  amplitude.  In these diagrams the Eikonal lines are the \emph{pegs}
  introduced above. Loops are then formed by connecting additional
  gluons between these Eikonal lines: these are the \emph{edges}
  mentioned above. Physically, the latter represent ``soft'' fields
  whose energy is small compared to energies of the external
  partons. Since each gluon emission along a given Eikonal line is
  associated with a non-commuting ${\rm SU}(N)$ matrix, the order of
  these emissions (corresponding to the \emph{height} of the edge on
  the peg in the above terminology) is important, and distinguishes
  between different web diagrams belonging to the same web world.
\end{itemize}
Following \cite{OTR} we consider in this paper a particular subset of
the web diagrams, those generated by any number of single gluon
exchanges between the Eikonal lines. We thus exclude from the present
discussion any web diagrams that includes 3 or 4 gluon vertices, as
well as ones where a gluon connects an Eikonal line to itself. 
The generalization to these cases is interesting and will be addressed by the authors in a future paper.

The present paper is a first combinatorial study of these web worlds
and proves several general results.  
One of our results provides a
rich connection between those colourings of a web diagram that lead to
a unique reconstruction of the diagram on one hand, and
order-preserving mappings of certain posets (partially ordered sets)
on the other.  This connection is then used to show that the diagonal
entries of a web-mixing matrix are obtained by summing certain
polynomials determined by the descents in permutations in the
Jordan-H\"older set of all linear extensions of the associated
poset.

As for results on the enumeration of web worlds we give two
tri-variate generating functions, recording statistics that keep track
of the number of pegs, the number of edges and the number of pairs of
pegs that are joined by some edge. 
We give a generating function for the number of proper web worlds
in terms of three different statistics.
We also obtain an expression for
the number of different web diagrams in a given web world, in terms of
entries of a matrix that represents the web world.

Three special cases of web worlds are examined in great detail, and
the traces of the web-mixing matrices calculated in each case.  The
first of these cases concerns web worlds whose web diagrams are
uniquely encoded by permutations, because each peg, except for the
last one, contains the left endpoint of a unique edge and all right
endpoints are on the last peg.  In this case we give exact values for
entries of the web-mixing matrices in terms of a permutation statistic
on the pair of permutations that index each entry.  The other two
cases are very different from the first, but quite similar to each
other.  Namely, these are web worlds where each pair of adjacent pegs
has a unique edge between them and there are no other edges.  They
differ in that in one case we look at the sequence of pegs cyclically
and demand that one edge join the first and the last peg.  Despite the
small difference in definition, the cyclic setup changes the trace of
the resulting web-mixing matrices greatly.

Generalising the last two cases just mentioned we define a class of
web worlds that we call {\it{transitive web worlds}}.  Remarkably, the
set of transitive web worlds is in one-to-one correspondence with the
$(2+2)$-free posets.  It seems likely that this connection will
engender some interesting results, given the fast growing literature
on these posets (see \cite{tpt,composition,djk,robert,remmel} for
references).

The outline of this paper is as follows.  Section~2 defines web
diagrams, web worlds and operations on web diagrams such as colourings
and how the partitioning of a web diagram induced by a colouring
describes the construction of another web diagram.  This section also
defines the web-colouring, $\MMmatrix{}{(x)}{}$, and web-mixing,
$\RRmatrix{}{}$, matrices. Section~3 looks at colourings of a
web diagram whose corresponding reconstructions result in the same
web diagram.  Section 4 gives generating functions for web worlds in
terms of three statistics; the number of pegs, the number of edges, and
the number of pairs of pegs that have at least one edge between
them. It also gives an expression for the number of different web
diagrams arising from a web world.  
Section~5 is an interlude on some material that will be used for 
determining web mixing matrices for particular classes of web worlds 
in Sections 6 to 8

In Section~6 we study a certain web world on $n+1$ pegs whose diagrams
are uniquely encoded by permutations of the set $\{1,\ldots,n\}$.
Sections 7 and 8 look at two web worlds with vastly different
properties from that of Section~6, but closely related to one another,
where each pair of adjacent pegs is connected by a unique edge.
Section~9 defines {\it{transitive web worlds}} which are shown to be in
one-to-one correspondence with $(2+2)$-free posets and thus with a
host of other families of combinatorial structures.  The paper ends
with some challenging open problems in Section~10.

\section{Web-diagrams}

Intuitively, a web diagram consists of a sequence of pegs and a set of
edges, each connecting two pegs, as illustrated in the example in
Figure~\ref{first:example}.  In the following formal definition of a
web diagram $D$, a 4-tuple $(x_j,y_j,a_j,b_j) \in D$ will represent an
edge whose left vertex has height $a_j$ on peg $x_j$ and whose right
vertex has height $b_j$ on peg $y_j$.

\begin{definition}
  A {\it{web diagram}} on $n$ pegs having $L$ edges is a collection
  $D=\{e_j=(x_j,y_j,a_j,b_j): 1\leq j \leq L\}$ of 4-tuples that
  satisfy the following properties:
  \begin{enumerate}
  \item[(i)] $1\leq x_j < y_j \leq n$ for all $j \in \{1,\ldots,L\}$.
  \item[(ii)] For $i \in \{1,\ldots,n\}$ let $\peg_i(D)$ be the number
    of $j$ such that $x_j$ or $y_j$ equals $i$, that is, the number of
    edges in $D$ incident with peg $i$.  Then $$\{b_j: y_j=i\} \cup
    \{a_j: x_j = i\} ~=~ \{1,2,\ldots,\peg_i(D)\}.$$
  \end{enumerate}
  We write $\Pegs(D)=(\peg_1(D),\ldots,\peg_n(D))$.  Condition (ii)
  says that the labels of the $\peg_i(D)$ vertices on peg $i$, when
  read, say, from top to bottom, are a permutation of the set
  $\{1,\ldots,\peg_i(D)\}$.
\end{definition}

Given a web diagram $D=\{e_j=(x_j,y_j,a_j,b_j): 1\leq j \leq L\}$, 
let 
\begin{align*}
\PegSet(D) &= \{x_1,\ldots,x_L,y_1,\ldots,y_L\},\\
\EdgeSet(D) &= D,\\
\PegpairsSet(D) &= \{(x_1,y_1),\ldots,(x_L,y_L)\}.
\end{align*}

\begin{example}
The web diagram given in Figure~\ref{first:example} is 
$$D=\{
(1,2,1,1),
(1,7,2,2),
(2,4,2,3),
(3,4,1,1),
(3,6,2,4),
(4,6,2,3),
(4,6,4,2),
(5,6,1,1),
(5,7,2,1)
\}.$$
The number of vertices on each peg is $\Pegs(D)=(2,2,2,4,2,4,2)$. 
Also, 
\begin{align*}
\PegSet(D)&=\{1,2,3,4,5,6,7\},\\
\PegpairsSet(D) &= \{
(1,2),
(1,7),
(2,4),
(3,4),
(3,6),
(4,6),
(5,6),
(5,7)\}.
\end{align*}
Notice that in this particular example $\PegpairsSet(D)$ has size one less than $\EdgeSet(D)$ since there are two edges that connect pegs 4 and 6.
\end{example}

\begin{figure}
\centerline{
    \begin{tabular}{ccc}
      \includegraphics[scale=0.8]{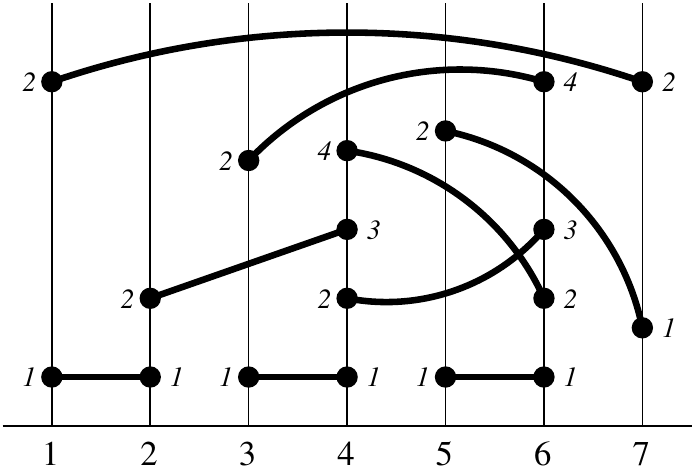} &&
      \includegraphics[scale=0.8]{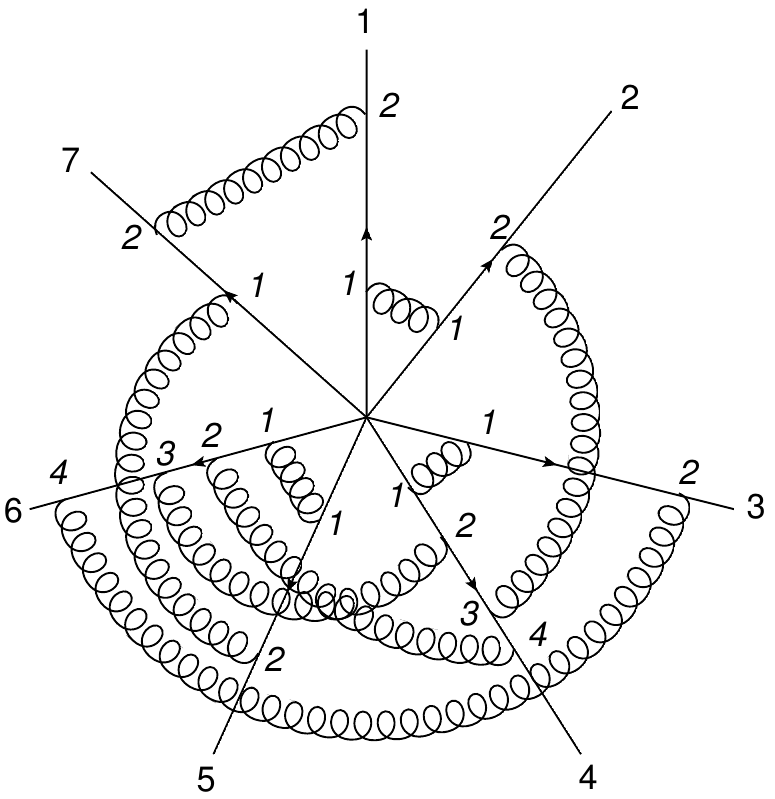}
    \end{tabular}
}
\caption{\label{first:example} In the diagram on the left the indices
  of the pegs are shown at the bottom.  The heights of the endpoints
  of the edges are shown in italics at each endpoint.  The unique edge between
  pegs 3 and 6 is represented by the 4-tuple $(3,6,2,4)$ since the
  left endpoint of the edge (on peg 3) has height 2 and the right
  endpoint of the edge (on peg 6) has height 4.  The diagram on the
  right is the Feynman diagram illustration of the web diagram.}
\end{figure}

We now define the sum of two web diagrams.
%
%
\begin{definition}
  Let $D = \{e_j=(x_j,y_j,a_j,b_j): 1\leq j \leq L\}$ and
  $D'=\{e'_j=(x'_j,y'_j,a'_j,b'_j): 1\leq j \leq L'\}$ be two
  web diagrams with $\PegSet(D),\PegSet(D') \subseteq \{1,\ldots,n\}$.  
  The {\textit{sum}} $D\oplus D'$ is the
  web diagram obtained by placing the diagram $D'$ on top of $D$;
$$
D\oplus D' ~=~ D \cup \{(x'_j,y'_j,a'_j + \peg_{x'_j}(D), b'_j +
\peg_{y'_j}(D)): 1\leq j \leq L'\}.
$$ 
If there exist two non-empty web diagrams $E$ and $F$ such that $D=E
\oplus F$ then we say that $D$ is {\it{decomposable}}.  Otherwise we
say that $D$ is {\it{indecomposable}}.
\end{definition}

\begin{example}\label{ex-diag-poset}
  Consider the following two web diagrams:
  $D_1=\{(1,4,1,1),(2,6,1,2),(2,6,2,1)\}$ and
  $D_2=\{(1,2,1,1),(3,5,1,1),(5,6,2,1)\}$.  For $D_1$ we have 
  $(\peg_1(D_1),\ldots,\peg_6(D_1))=(1,2,0,1,0,2)$ and so
  \begin{align*}
\lefteqn{D_1\oplus D_2} \\
  &= \{(1,4,1,1),(2,6,1,2),(2,6,2,1)\} ~ \cup ~ \left\{(1,2,1+\peg_1(D_1),1+\peg_2(D_1)),\right. \\
 & \quad \left. (3,5,1+\peg_3(D_1),1+\peg_5(D_1)),(5,6,2+\peg_5(D_1),1+\peg_6(D_1))\right\} \\
  &= \{(1,4,1,1),(2,6,1,2),(2,6,2,1)\} ~\cup ~ \{(1,2,1+1,1+2),(3,5,1+0,1+0),(5,6,2+0,1+2)\} \\
  &= \{(1,4,1,1),(2,6,1,2),(2,6,2,1),(1,2,2,3),(3,5,1,1),(5,6,2,3)\}.
\end{align*}
\end{example}

%
%
\begin{definition} \label{twofive} Let $D$ be a web diagram and $X
  \subseteq D$. Let $\rel(X)$ be the web diagram that results from
  re-labeling the third and fourth entries of every element of $X$ so
  that the set of labels of points on a peg $i$ of $\rel(X)$ are
  $\{1,2,\ldots,\ell_i\}$ for some $\ell_i$, and for all pegs.  We
  call $\rel(X)$ a {\it{subweb diagram}} of $D$.
\end{definition}

Note that we do not remove empty pegs.
The reason we do not remove them is that we will be defining an operation
which combines subweb diagrams of a diagram. 
This will be made clear in Definitions ~\ref{twonine} and \ref{rsk}.

\begin{example}
  Let $D$ be the web diagram in Figure~\ref{first:example}.  Let $$X =
  \{(1,7,2,2),(3,6,2,4),(4,6,2,3),(5,6,1,1)\}.$$ Then
  $\rel(X)=\{(1,7,1,1),(3,6,1,3),(4,6,1,2),(5,6,1,1)\}$.  See
  Figure~\ref{rel:example}.
\end{example}

\begin{figure}
  \begin{tabular}{ccc}
    \includegraphics[scale=0.8]{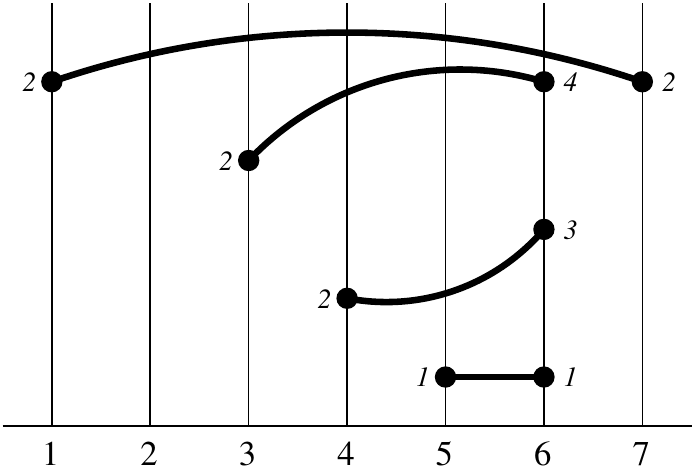} &&
    \includegraphics[scale=0.8]{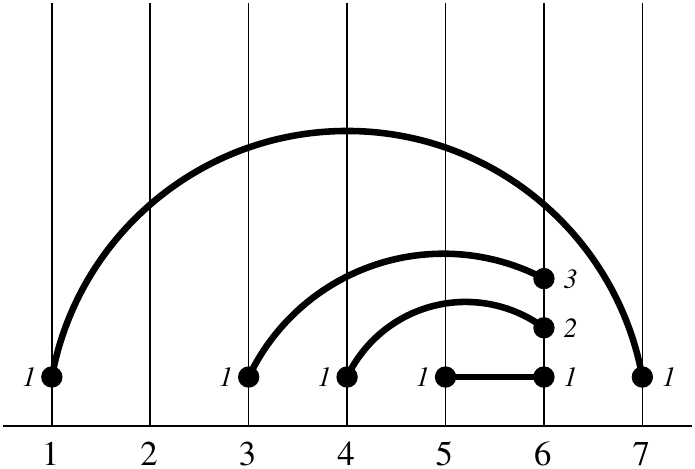} 
  \end{tabular}
  \caption{The transformation $X\to \rel(X)$\label{rel:example}.}
\end{figure}

%
%
\begin{definition} \label{twoseven} 
  Suppose that $D=\{(x_j,y_j,a_j,b_j): 1\leq j \leq L\}$ is a
  web diagram on $n$ pegs. Let $\Pegs(D) =
  (\peg_1(D),\ldots,\peg_n(D))$.  Let $\pi=(\pi^{(1)},\ldots,
  \pi^{(n)})$ be a sequence of permutations where $\pi^{(i)}$ is a
  permutation of the set $\{1,\ldots,\peg_i(D)\}$.  Let $\pi (D)$ be
  the web diagram that results from moving the vertex at height $j$ on
  peg $i$ to height $\pi^{(i)}(j)$, for all $j$ and $i$.  Finally let
  $\WW(D) = \{\pi(D): \pi \in \Pegs(D)\}$, the set of all possible web
  diagrams that can be obtained from $D$.  We call $\WW(D)$ the
  {\em{web world}} of $D$.
\end{definition}

%
%
\begin{example} 
  If $D=\{(1,2,1,2),(1,2,2,1)\}$ then $\WW(D)=\{D,E\}$ where
  $E=\{(1,2,1,1),(1,2,2,2)\}$.
\end{example}

In terms of physics, a particular subset of web worlds are of interest. 

\begin{definition} \label{graphweb}
Let $W$ be a web world and $D=\{(x_i,y_i,a_i,b_i) ~:~ 1\leq i \leq L\} \in W$.
Let $G(W)=(V,E,\ell)$ be the edge-labeled simple graph where $V=\PegSet(D)$,
\begin{align*}
E &= \left\{ \{x_1,y_1\},\{x_2,y_2\},\ldots,\{x_L,y_L\} \right\}\\
\noalign{and}
\ell(\{x,y\}) &= |\{1\leq i \leq L ~;~ x_i = \min(x,y) \mbox{ and } y_i = \max(x,y)\}|
\end{align*}
for all $\{x,y\} \in E$. We call $G(W)$ the {\it{web graph}} of the web world $W$.
\end{definition}

The web graph is the graph that results from `forgetting' the heights of all endpoints of edges in a web diagram. 
Another way to say this is that this is the graph one sees by looking down onto a web diagram, so that the pegs appear as points, and where each edge of $G(W)$ is weighted by the number of edges between the two pegs containing its endpoints in the original diagram.
For example, the web graph corresponding to the web world generated by the diagram in Figure~\ref{first:example} is:

\centerline{\includegraphics[scale=0.75]{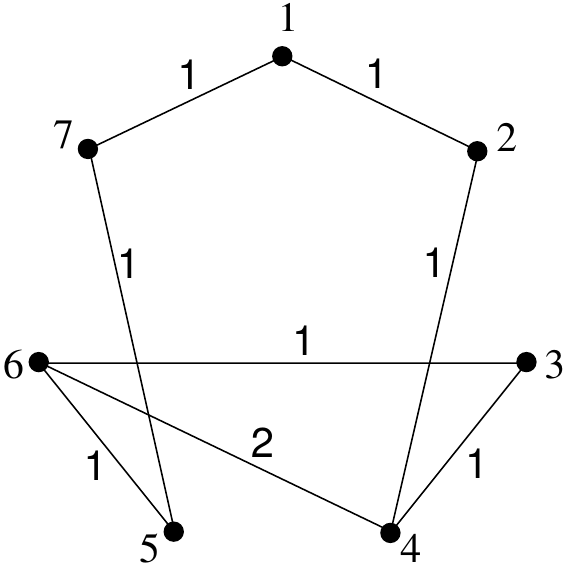}}

\begin{definition}
A web world $W$ is called {\it{proper}} if the web graph $G(W)$ is connected.
\end{definition}

Proper web worlds have been called `webs' in references 
~\cite{WIM,OTR,GPO,MSS}.
Their significance is that they contribute to the exponent of an Eikonal 
amplitude, while web worlds corresponding to disconnected web 
graphs do not. The three web worlds we look at later on is this paper are 
proper web worlds. We now introduce colouring and reconstruction 
operations on our web diagrams. 

%
%
\begin{definition} \label{twonine} Suppose that
  $D=\{e_i=(x_i,y_i,a_i,b_i): 1\leq i \leq L\}$ is a web diagram on
  $n$ pegs, and $\ell \leq L$ a positive integer.  A {\it{colouring}}
  $c$ of $D$ is a surjective function $c:\{1,\ldots,L\} \to
  \{1,\ldots,\ell\}$.  Let $D_c(j)= \{e_i \in D : c(i)=j\}$ for all
  $1\leq j \leq \ell$, the set of all those edges of $D$ that are
  coloured $j$.  The \emph{reconstruction $\Recon{D}{c}\in W(D)$ of
    $D$ according to the colouring $c$} is the web diagram
$$\Recon{D}{c} = \rel(D_c(1)) \oplus \rel(D_c(2)) \oplus \cdots \oplus \rel(D_c(\ell)).$$
\end{definition}

%
%
\begin{definition} \label{rsk}
  Let $D$ be a web diagram on $n$ pegs, and let $c$ be an
  $\ell$-colouring of $D$.  We say that the colouring $c$ is
  {\em{self-reconstructing}} if $\Recon{D}{c}=D$.
\end{definition}

%
%
Given $\WW=\WW(D)$ for some web diagram $D$ on $n$ pegs, and $D_1,D_2 \in \WW(D)$, let
$$F(D_1,D_2,\ell)=\{\mbox{$\ell$-colourings $c$ of $D_1$} ~:~ \Recon{D_1}{c}=D_2\}$$
and $f(D_1,D_2,\ell)=|F(D_1,D_2,\ell )|$.  Let $\MMmatrix{(W)}{(x)}{}$
be the matrix whose $(D_1,D_2)$-entry is 
$$\MMmatrix{(W)}{(x)}{D_1,D_2} = \sum_{\ell \geq 1} x^{\ell} f(D_1,D_2,\ell).$$
We call this matrix the {\it{web-colouring matrix}}.  From a physics
perspective, another matrix is of more immediate interest.  Let
$\RRmatrix{(W)}{}$ be the matrix whose $(D_1,D_2)$ entry is
\begin{align*}
  \RRmatrix{(W)}{D_1,D_2} &= \sum_{\ell \geq 1} \dfrac{(-1)^{\ell -1}}{\ell} f(D_1,D_2,\ell).
\end{align*}
We call $\RRmatrix{(W)}{}$ the {\it{web-mixing matrix}} of $W$.  It is
straightforward to show that the $(D_1,D_2)$ entries of
$\MMmatrix{(W)}{(x)}{}$ and $\RRmatrix{(W)}{}$ are related via the
following formula:
\begin{align}\label{calculateR}
  \RRmatrix{(W)}{D_1,D_2} &= \int_{-1}^{0} \dfrac{\MMmatrix{(W)}{(x)}{D_1,D_2}}{x} dx \;=\; -\int_0^1 \dfrac{\MMmatrix{(W)}{(-x)}{D_1,D_2}}{x} dx .
\end{align}

In the present paper, the basic problems we consider are as follows:
Given a web world $\WW$,
  \begin{enumerate}
  \item[(i)] What can we say about the matrices $\MMmatrix{(W)}{(x)}{}$ and 
	$\RRmatrix{(W)}{}$, their entries, trace and rank?
  \item[(ii)] Can we determine the entries of $\MMmatrix{(W)}{(x)}{}$ and $\RRmatrix{(W)}{}$
    for special cases?
  \end{enumerate}

In Gardi and White~\cite{GPO} it was shown that
\begin{theorem}\cite{GPO}\label{cethm}
  Let $\WW$ be a web world.
  \begin{enumerate}
  \item[(i)] The row sums of $\RRmatrix{(W)}{}$ are all zero.
  \item[(ii)] $\RRmatrix{(W)}{}$ is idempotent.
  \end{enumerate}
\end{theorem}

In the next section we look at diagonal entries of the matrices
$\MMmatrix{(W)}{(x)}{}$ and $\RRmatrix{(W)}{}$.  The following theorem
is the $\MMmatrix{}{}{}$-analogue to Theorem~\ref{cethm}(i). We omit
the proof as it is rather elementary.

\begin{theorem}
  Let $D$ be a web diagram with $m=|D|$ edges. The row sums
  $\MMmatrix{(W)}{(x)}{}$ are all the same and given by the 
  ordered Bell polynomials
  $$
  \aleph_m(x)=\sum_{\ell=1}^{m} x^{\ell} {m\brace \ell} \ell !
  $$
  where $\stirling{m}{\ell}$ are the Stirling numbers of the 2nd kind.
\end{theorem}

Using the recursion ${m+1\brace \ell} = {m\brace \ell-1}+ {m\brace \ell}$
in the definition of $\aleph_{m+1}(x)$ above, one finds that
the polynomials $\aleph_m(x)$ satisfy the differential equation
$x (d/dx)((x+1)\aleph_m(x)) = \aleph_{m+1}(x)$.  Equivalently,
$$
\sum_{m\geq 0} \aleph_m(x) t^m ~=~ \dfrac{1}{1 - \dfrac{xt}{1 -
    \dfrac{(x + 1)t}{1 - \dfrac{2xt}{1 - \dfrac{2(x + 1)t}{1 -
          \ddots}}}}}.
$$


\section{Self-reconstructing colourings and order-preserving maps}

In this section we will study those colourings of a web diagram $D$
that reconstruct $D$.  This is to gain insight into the diagonal
entries of the matrices $\MMmatrix{(W)}{(x)}{}$ and $\RRmatrix{(W)}{}$
and thus their traces.  
Since  $\RRmatrix{(W)}{}$ is idempotent (Theorem~\ref{cethm}), its trace is 
also its rank, which plays an important role in the physics 
context~\cite{WIM}.  In what follows we will see that every
web diagram $D$ can be decomposed and written as a sum of
indecomposable subweb diagrams.  A poset (partially ordered set) $P$
on the set of these indecomposable subweb diagrams is then formed.
Self-reconstructing colourings of $D$ are then shown to correspond to
linear extensions of $P$, to be explained below.
%
%
\begin{definition}
  Let $W$ be a web world and $D \in W$.  Suppose that $D=E_1 \oplus
  E_2 \oplus \cdots \oplus E_k$ where each $E_i$ is an indecomposable
  web diagram.  Define the partial order $P=(P,\preceq)$ as follows:
  $P=(E_1,\ldots ,E_k)$ and $E_i \preceq E_j$ if
  \begin{enumerate}
  \item[(a)] $i<j$, and
  \item[(b)] there is an edge $e=(x,y,a,b)$ in $E_i$ and an edge
    $e'=(x',y',a',b')$ in $E_j$ such that an endpoint of $e$ is below
    an endpoint of $e'$ on some peg.
  \end{enumerate}
We call $P(D)$ the \emph{decomposition poset of $D$}.
\end{definition}

\begin{example}
  Consider the diagram $D$ given in Figure~\ref{first:example}.  The
  poset $P(D)$ we get from this diagram is illustrated as follows:
  \ \\
  \centerline{\includegraphics{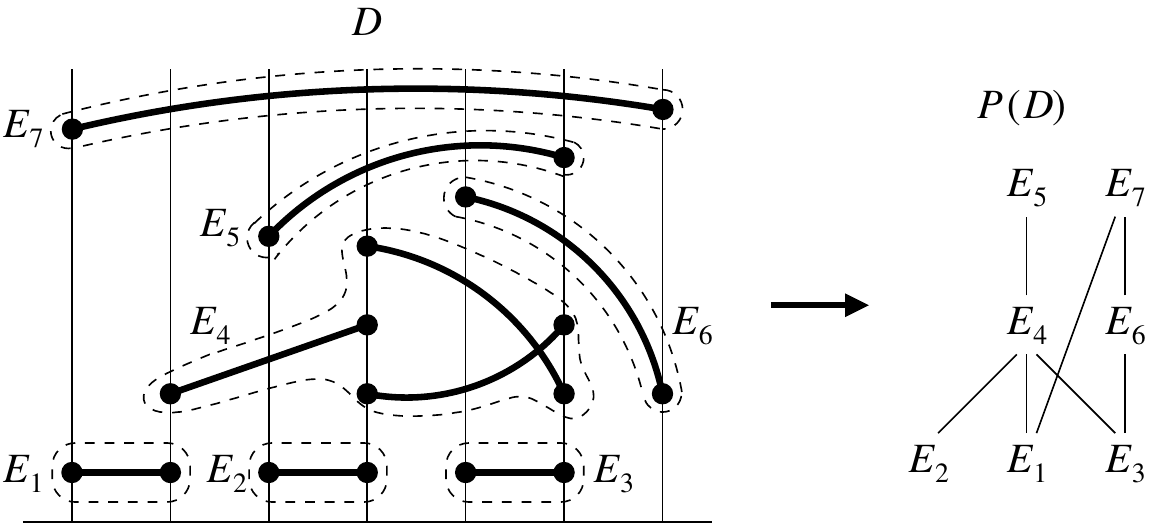}}
  \ \\
  Note that $E_1=\{(1,2,1,1)\}$, $E_2=\{(3,4,1,1)\}$,
  $E_3=\{(5,6,1,1)\}$, $E_4=\{(2,4,2,3),(4,6,2,3),$ $(4,6,4,2)\}$,
  $E_5=\{ (3,6,2,4)\}$, $E_6=\{ (5,7,2,1)\}$, and $E_7=\{(1,7,2,2)\}$.
\end{example}

Before we explain the relation between the linear extensions of $P(D)$
and self-reconstructing colourings we need some background.  Given any
two posets $P$ and $Q$, a map $f:P\to Q$ is \emph{order-preserving}
if, for all $x$ and $y$ in $P$, $x\preceq_P y$ implies $f(x) \preceq_Q
f(y)$.  Let $p$ be the number of elements in a poset $P$.  A
\emph{linear extension} of $P$ is an order-preserving bijection
$f:P\to [p]$ where $[p]=\{1,\ldots,p\}$ is equipped with the usual
order on the integers.  Each linear extension of $P$ can be
represented by a permutation of the elements of $P$, where the first
element in the permutation is that element $x$ of $P$ for which
$f(x)=1$ and so on.  The set of these permutations is called the
\emph{Jordan-H\"older set of $P$}, and denoted $\mathcal{L}(P)$.

Recall that a descent in a permutation $\pi=a_1a_2\ldots a_n$ is an
$i$ such that $a_i>a_{i+1}$, and let $\des(\pi)$ be the number of
descents in $\pi$. The first part of the following lemma is from \cite[Theorem 3.15.8]{ec1}.
The second part follows from the first part using the inclusion-exclusion principle.

\begin{lemma}\label{lemma-ec1}
  Let $P$ be a poset with $p$ elements, and let $\Omega(P,m)$ be the
  number of order preserving maps $\sigma:P \to [m]$.  Then
$$
\sum_{m\geq 0} \Omega(P,m) x^m = \dfrac{1}{(1-x)^{p+1}} \sum_{\pi \in
  \mathcal{L}(P)} x^{1+\des(\pi)}.
$$ 
Let $\Theta(P,m)$ be the number of surjective order-preserving maps
from $P$ to $[m]$.  Define
$\Theta(P,0)=\Omega(P,0)=0$. Then we have
$$
\Omega(P,m) = \sum_k {m \choose k} \Theta(P,k), \qquad \Theta(P,m) =
\sum_k {m \choose k}(-1)^{m-k} \Omega(P,k).
$$

\end{lemma}


From now on, we will assume that, as in Example~\ref{ex-diag-poset},
we have labeled the elements of $P=P(D)$ \emph{naturally}, that is, so
that if $E_i<E_j$ in $P$ then $i<j$.  In a permutation
$\pi=E_{i_1}E_{i_2}\ldots E_{i_p}$ in $\mathcal{L}(P)$, declare $k$ to
be a descent if and only if $i_k>i_{k+1}$.

\begin{theorem} \label{thm33} 
Let $D$ be a web diagram with $D=E_1\oplus \ldots\oplus E_k$ where the entries of the sum are all indecomposable web diagrams.
Let $P=P(D)$ and $p=|P(D)|$. If every member of the sequence $(E_1,\ldots,E_k)$ is distinct then
  \begin{align}
  \MMmatrix{(W)}{(x)}{D,D} &= \sum_{\pi \in \mathcal{L}(P)} x^{1+\des(\pi)} (1+x)^{p-1-\des(\pi)} \label{equationtwo}\\
  \noalign{and}
  \RRmatrix{(W)}{D,D} &= \sum_{\pi \in \mathcal{L}(P)} \dfrac{(-1)^{\des(\pi)}}{p{p-1 \choose \des(\pi)}}. \label{equationthree}
\end{align}
\end{theorem}

\begin{proof} 
  The number $f(D,D,\ell)$ defined in Section~2 is the number of
  surjective order-preserving maps from $P(D)$ to $\{1,\ldots,\ell\}$.
  By Lemma \ref{lemma-ec1} we get
$$
\sum_{m\geq 0} \Theta(P,m) x^m = (1+x)^{p} \sum_{\pi \in
  \mathcal{L}(P)} \left( \dfrac{x}{1+x} \right)^{1+\des(\pi)} =
\MMmatrix{(W)}{(x)}{D,D}.
$$ 

Let $B(x,y)=\int_0^1 t^{x-1} (1-t)^{y-1} dt$ be the beta function.
For the diagonal terms of the web mixing matrix we have
\begin{align*}
  \RRmatrix{(W)}{D,D} &= \sum_{\pi \in \mathcal{L}(P)} \int_{-1}^{0} dx \left( 
    x^{\des(\pi)} (1+x)^{p-1-\des(\pi)} 
  \right) \\
  &= \sum_{\pi \in \mathcal{L}(P)} (-1)^{\des(\pi)} B(\des(\pi)+1,p-\des(\pi)) \\
  &= \sum_{\pi \in \mathcal{L}(P)} \dfrac{(-1)^{\des(\pi)}}{p{p-1 \choose \des(\pi)}}.\qedhere
\end{align*}
\end{proof}

By Theorem~\ref{thm33}, computing the diagonal entries of our matrices
is thus equivalent to computing the descent statistic on the
Jordan-H\"older set of the corresponding poset.

\begin{example}
  Let $D$ be the following web diagram:
  \ \\[1em]
  \centerline{\includegraphics{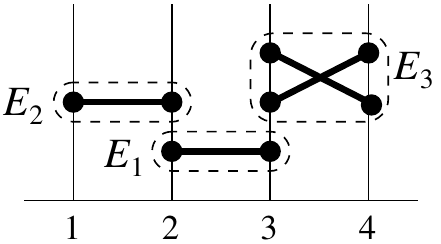}}
  \ \\
  Since each of the web diagrams $(E_1,E_2,E_3)$ are distinct, Theorem~\ref{thm33} may be applied.
  The poset $P=P(D)$ is the poset on $\{E_1,E_2,E_3\}$ with
  relations $E_1<E_2,E_3$.  We find that
  $\mathcal{L}(P)=\{E_1E_2E_3,E_1E_3E_2\}$, with $\des(E_1E_2E_3)=0$
  and $\des(E_1E_3E_2)=1$. Consequently we have 
$$\MMmatrix{(W(n))}{(x)}{D,D}
  = x(1+x)^2
  + x^2 (1+x) = x+3x^2+2x^3$$ and 
$\RRmatrix{(W)}{D,D} = {(-1)^0}/{3} +
  {(-1)^1}/{3{2\choose 1}} = 1/6$.
\end{example}

Given a web world $W$, let $\AllPosets(W) = \{P(D): D \in W\}$.
For $P \in \AllPosets(W)$, let $\multip{W}{P} = |\{D \in W ~:~ P(D)=P\}|$.
Using this notation we have
\begin{corollary} \label{woc}
Suppose that $W$ is a web world whose graph $G(W)=(V,E,\ell)$ is such that $\ell(e)=1$ for all $e \in E$. Then
\begin{align}
  \trace (\MMmatrix{(W)}{(x)}{}) \label{eastbound}
  &= \sum_{D \in W} \MMmatrix{(W)}{(x)}{D,D} \nonumber \\
  &= \sum_{P \in \AllPosets(W)} \multip{W}{P} \sum_{\pi \in \mathcal{L}(P)} x^{1+\des(\pi)} (1+x)^{|P|-1-\des(\pi)}\\
  \noalign{and}
  \trace (\RRmatrix{(W)}{}) \label{westbound}
  &= \sum_{D \in W} \RRmatrix{(W)}{D,D} \nonumber\\ 
  &= \sum_{P \in \AllPosets(W)} \multip{W}{P} \sum_{\pi \in \mathcal{L}(P)} \dfrac{(-1)^{\des(\pi)}}{p{p-1 \choose \des(\pi)}}
\end{align}
\end{corollary}

Note that, owing to the idempotence of web mixing matrices (Theorem~\ref{cethm}), 
the trace of a web mixing matrix is equal to its rank. 
A corollary is that the trace must be a positive integer number.  In the physics context of reference~\cite {WIM} this invariant represents the number of independent contributions to the exponent of an Eikonal scattering amplitude from the corresponding web world.

We illustrate the above calculations in the following two examples. In
the first example posets on a different number of elements
emerge. However, in the second example only posets on three elements
emerge. The second example is a special case of the web world that
will be discussed later in Case 2 in Section \ref{sec:7}.

\begin{example} \label{falkirk} 
The web world $W(D)$ of the web diagram $$D=\{(1,2,1,1),(2,3,2,1),(3,4,2,1)\}$$ contains four web diagrams. 
$\AllPosets(W)$ contains three posets;
the chain $\mathbf{3}$ arises twice; the wedge poset $\wedge$ arises once, as does the $V$-shaped poset $\vee$.  Thus
$\multip{W}{\mathbf{3}}=2$ and $\multip{W}{\wedge}=\multip{W}{\vee}=1$.
We have 
$\mathcal{L}(\mathbf{3})=\{(1,2,3)\}$ and
$\mathcal{L}(\vee) = \mathcal{L}(\wedge) =\{(1,2,3),(1,3,2)\}$.  
Thus,
$$
\trace(\MMmatrix{(W)}{(x)}{})= 2 (x^1 (1+x)^2) + 1(x^1(1+x)^2 +
x^2(1+x)) + 1(x^1(1+x)^2 + x^2(1+x)) = 6x^3+10x^2+4x
$$ 
and
$$
\trace(\RRmatrix{(W)}{}) = 2(1/3) + 1(1/6) + 1(1/6) =1.
$$ 
(See Figure~\ref{burp}.)
\end{example}

\begin{figure} 
\newcolumntype{S}{>{\centering\arraybackslash} m{.3\linewidth} }
\begin{tabular}{|S|S|S|} \hline Diagram $D$ & Poset $P(D)$ &
  $\mathcal{L}(P(D))$ \\ \hline\hline
  \includegraphics{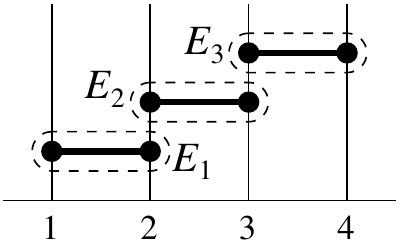} & \includegraphics{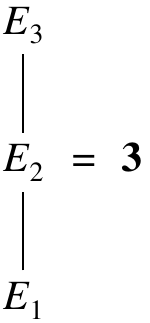} &  $\{(1,2,3)\}$ \\ \hline
  \includegraphics{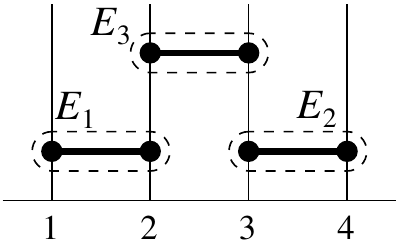} & \includegraphics{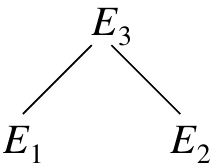} &  $\{(1,2,3), (1,3,2)\}$ \\ \hline
  \includegraphics{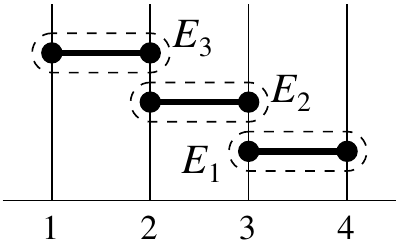} & \includegraphics{3chain.pdf} &  $\{(1,2,3)\}$ \\ \hline
  \includegraphics{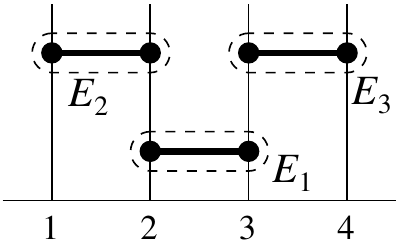} & \includegraphics{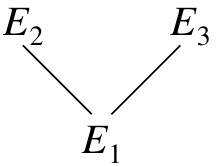} &  $\{(1,2,3),(1,3,2)\}$ \\ \hline
\end{tabular}
\caption{\label{burp} The four web diagrams in the web world generated
  from $D$ in Example~\ref{falkirk}.}
\end{figure}

\section{The number of web worlds}
In this section we present some results on the number of web worlds
with prescribed sizes of the sets $\PegSet$, $\EdgeSet$, and
$\PegpairsSet$. 
If $W$ is a web world and $D_1,D_2$ are two web diagrams in $W$, then
$\PegSet(D_1)=\PegSet(D_2)$, $\EdgeSet(D_1)=\EdgeSet(D_2)$, and 
$\PegpairsSet(D_1)=\PegpairsSet(D_2)$. We will use $\PegSet(W)$, 
$\EdgeSet(W)$, and $\PegpairsSet(W)$ to refer to these numbers without
having to refer to diagrams of the web world.

A web world $W$ is uniquely specified by its web graph $G(W)=(V,E,\ell)$.
An equivalent specification is by a square matrix
$\Represent(W)$ of integers whereby
$$\Represent(W)_{i,j} = 
\left\{ 
\begin{array}{ll}
\ell(i,j) & \mbox{ if } i<j \mbox{ and } \{i,j\} \in E, \\
0 & \mbox{ otherwise}.
\end{array}
\right.$$
Theorem~\ref{thm-nww-gf} gives the generating
function for the number of web worlds according to these three
statistics.  Web-worlds which contain pegs that have no incident
edges, equivalent to isolated vertices when one talks of graphs, have
little relevance in the corresponding physics model.  
Theorem~\ref{gin}
gives the generating function for the number of web worlds with no
isolated pegs, according to the three statistics.  
Theorem~\ref{sam} gives the generating function for the number of proper
web worlds according to the same three statistics as Theorem~\ref{gin}.
Theorem~\ref{allo} gives
an expression for the number of different diagrams in a web world in
terms of the values in the matrix $\Represent(W)$.

Let $W$ be a web world and $D$ a web diagram in $W$.  Define the
{\it{web diagram matrix}} of a web diagram $D$ to be a matrix
$\wdm(D)$ whose entries are sets:
\ \\[1em]
\centerline{$(a_j,b_j) \in \wdm(D)_{x_j , y_j}$ for all
  $(x_j,y_j,a_j,b_j) \in D$.}
\begin{example} \label{wmd_example} Consider the web diagram $D$ in
  Figure~\ref{first:example}. The set of 4-tuples for that web diagram
  is $\{ (1,2,1,1), (1,7,2,2), (2,4,2,3), (3,4,1,1), (3,6,2,4),
  (4,6,2,3), (4,6,4,2), (5,6,1,1), (5,7,2,1) \}.$ The web diagram
  matrix $\wdm(D)$ is:
$$\wdm(D) = \left(
\begin{matrix}
  \emptyset & \{(1,1)\} & \emptyset & \emptyset &\emptyset &\emptyset & \{(2,2)\} \\
  \emptyset & \emptyset & \emptyset & \{(2,3)\} & \emptyset & \emptyset & \emptyset \\
  \emptyset & \emptyset & \emptyset & \{(1,1)\} & \emptyset & \{(2,4)\} & \emptyset \\
  \emptyset & \emptyset & \emptyset & \emptyset & \emptyset & \{(2,3),(4,2)\} & \emptyset \\
  \emptyset & \emptyset & \emptyset & \emptyset & \emptyset & \{(1,1)\} & \{(2,1)\} \\
  \emptyset & \emptyset & \emptyset & \emptyset & \emptyset & \emptyset & \emptyset \\
  \emptyset & \emptyset & \emptyset & \emptyset & \emptyset &
  \emptyset & \emptyset
\end{matrix}
\right).$$
\end{example}

Every web world $W$ can be
uniquely specified by a matrix $\Represent(W)$ since all permutations of line
end-points on pegs form diagrams in the same web world, and it is only
the {\it{number}} of such lines that defines the web world. A matrix
$A=\Represent(W)$ is the matrix of a web world iff $a_{ij}=0$ for all
$1\leq j \leq i\leq m$ and the other entries are non-negative
integers.  These matrices are also obtained by simply taking the
cardinality of the entries of the matrices $\wdm$.

\begin{example}
  Let $W$ be the web world given in Example~\ref{wmd_example}.  Then
$$\Represent(W) = \left(
\begin{matrix}
  0 & 1 & 0 & 0 & 0 & 0 & 1 \\
  0 & 0 & 0 & 1 & 0 & 0 & 0 \\
  0 & 0 & 0 & 1 & 0 & 1 & 0 \\
  0 & 0 & 0 & 0 & 0 & 2 & 0 \\
  0 & 0 & 0 & 0 & 0 & 1 & 1 \\
  0 & 0 & 0 & 0 & 0 & 0 & 0 \\
  0 & 0 & 0 & 0 & 0 & 0 & 0
\end{matrix}
\right).$$
\end{example}

Let 
$$\nww(m,t,n)=\left| \left\{ \mbox{web worlds }W ~:~ 
\begin{array}{lcc}
\PegSet(W) &\subseteq& \{1,\ldots,m\},\\ |\EdgeSet(W)|&=&t,\\ |\PegpairsSet(W)|&=&n 
\end{array}
\right\}\right|,$$
the {\ul{n}}umber of {\ul{w}}eb {\ul{w}}orlds with the three prescribed properties.
The generating function $\NWW(u,z,y)$ for these numbers is
\begin{theorem}\label{thm-nww-gf}
  $\NWW(u,z,y) ~=~ \displaystyle\sum_{m\geq 2\atop t,n\geq 0}
  \nww(m,t,n) u^m z^t y^n ~=~ \sum_{m \geq 2} \left( 1+
    \dfrac{yz}{1-z}\right)^{m\choose 2} u^m.$
\end{theorem}

The above generating function allows pegs which have no incident
edges. We now consider web diagrams which do not contain such pegs.
Let 
$$\nwwnip(m,t,n)=\left| \left\{ \mbox{web worlds }W ~:~ 
\begin{array}{lcl}
\PegSet(W) &=& \{1,\ldots,m\},\\ |\EdgeSet(W)|&=&t,\\ |\PegpairsSet(W)|&=&n 
\end{array}
\right\}\right|, $$
the {\ul{n}}umber of {\ul{w}}eb {\ul{w}}orlds having {\ul{n}}o {\ul{i}}solated 
{\ul{p}}egs and the three prescribed properties.
This collection of web worlds
corresponds to matrices $A = \Represent(W)$ with the added property
that the `hook' at position $(i,i)$ in $A$ is nonempty, i.e. for all
$1\leq i \leq m$,
$$a_{1,i}+\cdots + a_{i-1,i}+a_{i,i}+a_{i,i+1}+\cdots+a_{im} >0.$$
Let
$$\NPWW(u,z,y) = \sum_{m \geq 2 \atop t,n\geq 1} \nwwnip(m,t,n) u^m z^t y^n.$$
We have

\begin{theorem} \label{gin}
  $\NPWW(u,z,y) = - \left(\dfrac{u}{1+u}\right)^2 + \dfrac{1}{1+u}
  \displaystyle\sum_{m \geq 2} \left(
    1+\frac{yz}{1-z}\right)^{m\choose 2} \left(
    \frac{u}{1+u}\right)^m$.
\end{theorem}

\begin{proof}
  Every matrix that is counted by $\nwwnip(m,t,n)$ gives rise to ${m+i
    \choose i}$ matrices of dimension $m+i$ where we have inserted $i$
  zeros on the diagonal in appropriate places.  Fill in the hooks of
  these new zeros on the diagonal with zeros.  The resulting matrices
  are simply the collection of matrices which have integer entries,
  and whose diagonal and lower diagonal entries are all zero.
  Therefore
  \begin{align*}
  \sum_{m \geq 2 \atop t,n \geq 1} \sum_{i \geq 0} {m+i \choose i} \nwwnip(m,t,n) x^{m+i} z^t y^n 
  &= \sum_{m \geq 2\atop n,t \geq 1} \nww(m,t,n) x^m z^t y^n.
\end{align*}
By Theorem~\ref{thm-nww-gf} the right hand side can be written as
$$\sum_{m \geq 2} x^m \left(  \left(1+\dfrac{yz}{1-z}\right)^{m \choose 2} -1 \right).$$
The left hand side of the above equation is $\dfrac{1}{1-x}
\NPWW(x/(1-x),z,y)$.  Change variable by setting $u=x/(1-x)$ and the
result follows.
\end{proof}

Using standard techniques to extract the coefficient of $u^a z^b y^c$
in Theorem~\ref{gin} we have the following corollary.
\begin{corollary}
  $\nwwnip(a,b,c) =\displaystyle {b-1\choose c-1} \sum_{k} {a \choose k}
  {{k\choose 2} \choose c} (-1)^{a-k}$.
\end{corollary}

Let  
$$\npww(m,t,n)=\left| \left\{ \mbox{web worlds }W ~:~ 
\begin{array}{lcl}
\PegSet(W) &=& \{1,\ldots,n\},\\ |\EdgeSet(W)|&=&m,\\ |\PegpairsSet(W)|&=& t
\end{array}
\right\}\right|,$$
the {\ul{n}}umber of {\ul{p}}roper {\ul{w}}eb {\ul{w}}orlds with the three prescribed properties.
Let 
$$C(x,q,z)  ~=~ \sum_{m,t,n} \npww(m,t,n) x^m q^t \dfrac{z^n}{n!}.$$

\begin{theorem} \label{sam}
$C(x,q,z)~=~ \log\left(  
1+\ds\sum_{n\geq 1} \left(1+\dfrac{qx}{1-x}\right)^{n\choose 2} \dfrac{z^n}{n!}
\right).$
\end{theorem}

\begin{proof}
Using the exponential formula, the bivariate generating function 
$C(q,z)=\sum_{n,k} c(k,n) q^k \dfrac{x^{n}}{n!}$ where $c(k,n)$ is the number of connected simple graphs on $n$ labeled vertices having $k$ edges is seen to satisfy
$$\exp C(q,z) = 1+\sum_{n\geq 1} (1+q)^{n\choose 2} \dfrac{z^n}{n!}.$$
One may now replace $q$ by $q(x+x^2+x^3+\ldots)$ so that the power of $x$ records the sum of weights on labeled edges between vertices.
Therefore
$$\exp C(x,q,z) = 1+\sum_{n\geq 1} (1+q(x+x^2+\ldots))^{n\choose 2} \dfrac{z^n}{n!}.$$
Taking logarithms of both sizes we have:
$C(x,q,z)~=~ \log\left(  
1+\ds\sum_{n\geq 1} \left(1+\dfrac{qx}{1-x}\right)^{n\choose 2} \dfrac{z^n}{n!}
\right).$
\end{proof}

Next we present a different type of enumerative result. Rather than giving another theorem for the number of web worlds according to
various statistics, we now present a formula for the number of web diagrams contained in a given web world.

\begin{theorem} \label{allo}
  Let $W$ be a web world on $n$ pegs and $A=\Represent(W)$. The number
  of different diagrams $D \in W$ is
$$|W| ~=~ \prod_{i=1}^{n} {(a_{i*}+a_{*i})!} ~\Big/ \prod_{1\leq i<j\leq n} a_{ij}!$$
where $a_{i*}$ (resp. $a_{*i}$) is the sum of entries in column
(resp. row) $i$ of $A$.
\end{theorem}

\begin{proof}
  The number of diagrams $D \in W$ is the number of ways to form an
  $n\times n$ matrix $X=\wdm(D)$ whose entries consist of sets of
  ordered pairs that satisfy the following hook property for all
  $1\leq i \leq n$: The set $A \cup B$ is a permutation of the set
  $\{1,\ldots,x_{1i}+\cdots+x_{ii}+\cdots +x_{in}\}$ where $A$ is the
  set of all first elements of pairs in the sets
  $X_{i,i+1},\ldots,X_{i,n}$ and $B$ is the set of all second elements
  of pairs in the sets $X_{1,i},X_{2,i},\ldots,X_{i-1,i}$.

  The sets in each matrix entry can be replaced with the sequence of
  pairs wherein the pairs of elements are listed in increasing order
  of their first entry.

  The entries of the matrix $A=\Represent(W)$ tells us the
  cardinalities of the entries of $X$, i.e. $a_{ij}=|X_{ij}|$.  We
  count the number of ways to fill the entries with respect to the
  hook $X_{1,i},\ldots,X_{i,i},\ldots,X_{i,n}$.  Entries in the
  vertical part of the hook get inserted as second entries of the
  pairs.  Entries of the horizontal part of the hook get inserted as
  the first entries of pairs.  There are a total of
  $(a_{1,i}+\dotsb +a_{i,i}+\dotsb + a_{i,n})! = (a_{*i}+a_{i*})!$ ways
  to to this.  However we must divide by $a_{i,i+1}! \dotsb a_{i,n}!$
  as the sequence of first entries at each matrix entry must be
  increasing.  The result follows by multiplying together the number
  of ways of filling each of the $n$ hooks on the diagonal.
\end{proof}

\section{A brief interlude on colouring sequences}

Before looking at three examples of web worlds we collect here some
notation and terminology that will be used in subsequent sections.
Most of the notation and terminology concern colourings and we remind
the reader that constructing new web diagrams from a given web diagram
by means of a colouring was defined in Definition~\ref{twonine}.

Let
$$\Colours{n}{k} = \{ (c(1),\ldots,c({n})) ~:~ \{c(1),\ldots,c({n})\}=\{1,\ldots,k\}\}.$$
The set $\Colours{n}{k}$ is the set of all colour sequences of length
$n$ where each of the colours $\{1,\ldots,k\}$ appears at least once.
When performing calculations with respect to colour sequences, we will
find it useful to decompose a colour sequence as a list of colours as
they appear from left to right in the colour sequence (so that no two
adjacent colours are the same) along with a sequence which gives the
multiplicity of each unique colour.

Let $\Keys{n}{k}{} = \{\vec{c} \in \Colours{n}{k} : c(i) \neq c(i+1)
\mbox{ for all } 1\leq i< n\}$, the set of all colour sequences of
length $n$ which contain no two adjacent colours.  Define
\begin{align*}
  \Keys{n}{k}{=} &= \{ c \in \Keys{n}{k}{} ~:~ c(1)=c(n)\}\\
  \Keys{n}{k}{\neq} &= \{ c \in \Keys{n}{k}{} ~:~ c(1)\neq c(n)\}.
\end{align*}
This partition of $\Keys{n}{k}{}$ will be of interest in
Section~\ref{casethree} when we consider the colouring of an edge
between the first peg and last peg in a web diagram.

\begin{lemma} \label{keyslemma}
Let $n,k\in\N$ with $1\leq k \leq n$. Then
  \begin{align*}
    |\Keys{n}{k}{}| &= \sum_{i=0}^k (-1)^{k-i} {k \choose i} i (i-1)^{n-1}\\
    |\Keys{n}{k}{\neq}| &= \sum_{i=0}^k (-1)^{k-i} {k \choose i} \left( (i-1)^{n} +(i-1)(-1)^n \right)\\
    |\Keys{n}{k}{=}| &= \sum_{i=0}^k (-1)^{k-i} {k \choose i} \left( (i-1)^{n-1} +(i-1)(-1)^{n-1} \right).
  \end{align*}
\end{lemma}

\begin{proof}
  Since $n(n-1)^{k-1} = \sum_{i=0}^k {k \choose i} |\Keys{n}{i}{}|$
  for all $0\leq k \leq n$, inclusion-exclusion gives $|\Keys{n}{k}{}|
  = \sum_i (-1)^{k-i} {k \choose i} i(i-1)^{n-1}$ for all $0\leq k
  \leq n$.  Define $\NoEquals_{n,k} = \{(a_1,\ldots,a_n) ~:~ a_1\leq
  a_2,a_2\neq a_3,\ldots,a_m \neq a_1 \mbox{ and } a_i \in
  \{1,\ldots,k\}\}$.  This set is enumerated $|\NoEquals_{n,k}| =
  (k-1)^n+ (k-1)(-1)^n$.  We have
$$
\Keys{n}{k}{\neq} = \{a \in \NoEquals_{n,k} ~:~ \{a_1,\ldots,a_n\} =
\{1,\ldots ,k\}\}.
$$
Since $|\NoEquals_{n,k}| = \sum_{\ell} {k \choose \ell}
|\Keys{n}{\ell}{\neq}|$ for all $0\leq k \leq n$, inclusion-exclusion
yields
$$
|\Keys{n}{k}{\neq}| = \sum_i (-1)^{k-i} {k \choose i}
|\NoEquals_{n,i}| = \sum_i (-1)^{k-i} {k \choose i} \left( (i-1)^{n}
  +(i-1)(-1)^n \right),
$$ 
hence the second equation.  Subtracting the second equation from the
first equation in the statement of the Lemma yields the third
equation.
\end{proof}

Every $c \in \Colours{n}{k}$ can be uniquely expressed as a pair
$\langle w,A\rangle $, where $A= \{2\leq i \leq n: c(i) = c({i-1})\}
\subseteq [2,n]$, and $w=(c(i))_{i \geq 1 \atop i \not\in A} \in
\Keys{m}{k}{}$.  For example, if $c=(4,2,1,1,1,2,3,3,3,4,4,1)$ then we
have $A=(4,5,8,9,11)$ and $w=(4,2,1,2,3,4,1)$.  Notice that if
$c=\langle w,A\rangle $ where $c \in \Colours{n}{k}$ where $|A|=n-m$
and $w \in \Keys{m}{k}{}$, then $$\equ(c) ~:=~ |\{1\leq i < n :
c(i)=c(i+1)\}| ~=~ n-m.$$

\section{Case 1: A web world with $\Pegs(D)=(1,1,\ldots,1,n)$ }
Fix $n \in \N$ and let $W(n)$ be the web world of all diagrams
\begin{align}
  D&=D_{\pi} = \{ e_i= (\pi(i),n+1,1,i)~:~ 1\leq i \leq n\} \label{deepi}
\end{align}
where $\pi$ is a permutation of $\{1,\ldots,n\}$. 
An example of such a web diagram for $n=6$ is illustrated in Figure~\ref{caseoneexample}.

\begin{figure}
  \includegraphics{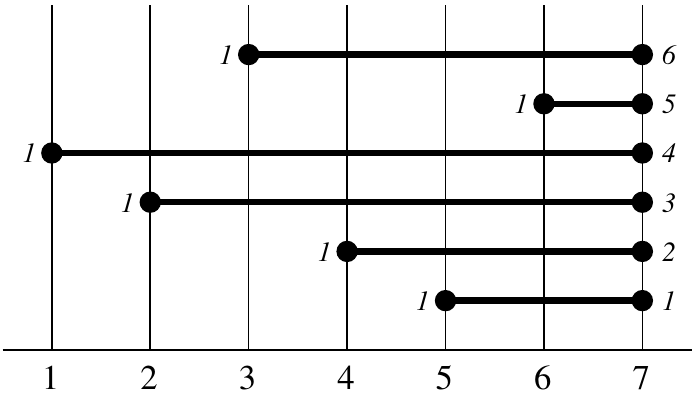}
  \caption{\label{caseoneexample} Example of a web diagram in a
    $\Pegs(D) = (1,1,1,1,1,1,6)$ web world.  This is represented by
    $D_{\pi}$ in Equation~(\ref{deepi}) with $\pi=(5,4,2,1,6,3)$.}
\end{figure}

\begin{definition}
  Given $c \in \Colours{n}{k}$, let $\alpha_c$ be the
  lexicographically smallest permutation in $\sn$ such that
$$c({\alpha_c(1)}) \leq c({\alpha_c(2)}) \leq \cdots \leq c({\alpha_c(n)}).$$
\end{definition}

Another way to read $\alpha_c$ from $c$ is to first list the positions
(in increasing order) at which 1 appears in $v$, then do the same for
2, and so on.

\begin{example}
  If $c=(3,1,1,2,1,2,2,3) \in \Colours{8}{3}$ then
  $\alpha_c=(2,3,5,4,6,7,1,8)$.
\end{example}

\begin{lemma}
  Given $D_{\pi} \in W(n)$ and $c \in \Colours{n}{k}$,
  $\Recon{D_{\pi}}{c} = D_{\pi \circ \alpha_c}$.
\end{lemma}

\begin{proof}
  Let $D_{\pi} \in W(n)$ and $c \in \Colours{n}{k}$.  Notice that
  edges can only share an endpoint on peg $n+1$. For this reason we
  have
$$D_{\pi} = 
(\pi(1),n+1,1,1) \oplus (\pi(2),n+1,1,1) \oplus \cdots \oplus
(\pi(n),n+1,1,1) .
$$
For all $j \in \{1,\ldots,k\}$ let $X_j=\{i:c(i)=j\}$ and define $Y_j$
to be the sequence of elements of $X_i$ written in increasing order.
Let $Z_j$ be the sequence of elements of $Y_1:Y_2:\cdots :Y_j$ where
$:$ denotes sequence concatenation and define $Z_0 =\emptyset$.  From
Definitions~\ref{twofive} and \ref{twonine} we have for all $j \in
\{1,\ldots,k\}$,
$$\rel (D_c(j)) = \bigoplus_{i=|Z_{j-1}|}^{|Z_j|} (\pi(\alpha_c(i)),n+1,1,1).$$
Thus
$$
\Recon{D_{\pi}}{c} ~=~ \bigoplus_{i=1}^{|Z_k|} (\pi(\alpha_c(i)),
n+1,1,1) ~=~ D_{\pi \circ \alpha_c}.\qedhere
$$
\end{proof}

\begin{example}
  Given $\pi = (2,8,5,4,1,3,7,6) \in \sss_8$ and $c=(3,1,1,2,1,2,2,3)
  \in \Colours{8}{3}$, we have $\alpha_c=(2,3,5,4,6,7,1,8)$ and $\pi
  \circ \alpha_c = (8,5,1,4,3,7,2,6)$.
\end{example}

For every pair $D_{\pi}$ and $D_{\sigma}$ of diagrams in $W(n)$, there
always exists a colouring of $D_{\pi}$ which will yield $D_{\sigma}$.
There is a unique colouring with a minimal number of colours which
achieves this and is described as follows:

\begin{definition}
  Let $\Minimal(\pi,\sigma)=(X_1,\ldots,X_{\ell})$ be the sequence of
  sequences obtained by repeatedly reading $\pi$ from left to right 
  until all the letters have been read, in the order in which they
  appear in $\sigma$.  
Define $\minimal(\pi,\sigma)=\ell$ and
  if $\pi_j \in X_i$ then set $c(j) = i$.
\end{definition}

\begin{example}
  If $\pi=(2,8,5,4,1,3,7,6)$ and $\sigma=(8,5,1,4,3,7,2,6)$ then we
  have $\Minimal(\pi,\sigma)=((8,5,1),(4,3,7),(2,6))$.  This means
  $\minimal(\pi,\sigma)=3$ and the unique colouring having fewest
  colours which transforms $D_{\pi}$ into $D_{\sigma}$ is
  $c=(1,3,2,2,1,3,2,1)$.
\end{example}

\begin{lemma}\label{lemmaone}
  Let $D_{\pi},D_{\sigma} \in W(n)$ with $\minimal(\pi,\sigma)=m$.
  Given $k$ with $m\leq k \leq n$, the number of $c \in
  \Colours{n}{k}$ for which $\pi \circ \alpha_c=\sigma$ is
  $f(D_{\pi},D_{\sigma},k)={n-m \choose k-m}$.
\end{lemma}

\begin{proof}
  Let $D_{\pi},D_{\sigma} \in W(n)$ with $\minimal(\pi,\sigma)=m$ and
  $\Minimal(\pi,\sigma)=(X_1,\ldots,X_{m})$. Let $x_i=|X_i|$.  Once
  the minimal colouring for converting $\pi$ to $\sigma$ is given, all
  other colourings easily follow from this.  All elements in $X_i$
  have colour $i$ and $\sigma$ is the permutation $X_1X_2\cdots X_m$.
  We may partition each of these sequences $X_i$ individually, to
  $(Y_1,\ldots,Y_k)$ whereby $Y_1\cdots Y_{\beta(1)}=X_1$,
  $Y_{\beta(1)+1}\cdots Y_{\beta(2)}=X_2$, and so on.  The number of
  ways of doing this is ${(x_1-1)+\ldots+(x_m-1) \choose (k-1)-(m-1)}$
  since there are $m-1$ positions already specified from the minimal
  colouring.  This value is ${n-m \choose k-m}$.
\end{proof}

\begin{theorem} 
  Suppose that $D_{\pi},D_{\sigma} \in W(n)$ with
  $m=\minimal(\pi,\sigma)$.  Then
  \begin{align}
  \MMmatrix{(W(n))}{(x)}{D_1,D_2} = x^m (1+x)^{n-m} 
  &\quad\mbox{ and }\quad
  \RRmatrix{(W(n))}{D_{\pi},D_{\sigma}} = 
  \dfrac{(-1)^{m-1}}{n {n-1\choose m-1}}.
\end{align}
Consequently, 
$$\begin{array}{lcl@{\qquad}lcl}
\RRmatrix{(W(n))}{D_{\pi},D_{\pi}} &=& 1/n,  & \trace\left(\RRmatrix{(W(n))}{}\right) &=& (n-1)! \\[0.75em]
\MMmatrix{(W(n)}{(x)}{D_{\pi},D_{\pi}}&=&x(1+x)^{n-1}, & \trace\left(\MMmatrix{(W(n)}{(x)}{}\right) &=& n!x(1+x)^{n-1}.
\end{array}$$
\end{theorem}

\begin{proof}
  Let $m=\minimal(\pi,\sigma)$.  Using Lemma~\ref{lemmaone} we have
  \begin{align}
  \MMmatrix{(W(n))}{(x)}{D_1,D_2} &= \sum_{k=m}^n x^k f(D_{\pi},D_{\sigma},k) = \sum_{k=m}^n x^k {n-m \choose k-m} = x^m (1+x)^{n-m}.
\end{align}
Using Equation~(\ref{calculateR}) we have
\begin{align*}
  \RRmatrix{(W(n))}{D_{\pi},D_{\sigma}} 
  &= -\int_0^1 \dfrac{\MMmatrix{(W(n)}{(-x)}{D_{\pi},D_{\pi}}}{x}\, dx \\
  &= (-1)^{m-1} \int_0^1 x^{m-1} (1-x)^{n-m}\, dx \\
  &= (-1)^{m-1} B(m,n-m+1), \\
  &= \dfrac{(-1)^{m-1}}{n {n-1\choose m-1}},
\end{align*}
where $B$ is the beta function.
Since $\minimal(\pi,\pi)=1$ we have that
$\RRmatrix{(W(n))}{D_{\pi},D_{\pi}} = 1/n$ and
$\MMmatrix{(W(n)}{(x)}{D_{\pi},D_{\pi}}=x(1+x)^{n-1}$.  Therefore
$\trace(\RRmatrix{(W(n))}{})=n!/n=(n-1)!$ and
$\trace(\MMmatrix{(W(n)}{(x)}{}=n!x(1+x)^{n-1}$.
This result is consistent with Equation~(\ref{equationthree}) as there is only one linear extension of the poset in this setting.
\end{proof}

Let $\sigma \in \sn$. It is straightforward to check that
$\minimal(\mbox{id},\sigma)=1+\des(\sigma)$.  Thus the number of
$\sigma$ in the top row of $\MMmatrix{(W(n))}{(x)}{}$ for which
$\minimal(\mbox{id},\sigma)=k$ is the Eulerian number $\left\langle {
    n \atop k}\right\rangle$.  The same observation yields that for
any $\pi \in \sn$, the number of $\sigma \in \sn$ such that
$\minimal(\pi,\sigma)=k$ is the Eulerian number $\left\langle { n
    \atop k} \right\rangle$.  This fact now accounts for value
multiplicities in the rows and columns of $\MMmatrix{(W(n))}{(x)}{}$
and $\RRmatrix{(W(n))}{}$.

\section{Case 2: A web world with $\Pegs(D)=(1,2,\ldots,2,1)$}\label{case2}
\label{sec:7}
Fix $n \in \mathbb{N}$ and let $W(n)$ be the collection of all web diagrams
\begin{align}\label{casetwotype}
  D &= \left\{ e_i = (i,i+1,x_i,y_{i+1}) ~:~ 1 \leq i \leq n+1  \right\}
\end{align}
where $x_1=y_{n+2}=1$ and $\{x_i,y_i\} = \{1,2\}$ for all $2 \leq i
\leq n+1$.  $W(n)$ is the set of all web diagrams consisting of $n+2$
pegs, $n+1$ lines whose endpoints are on adjacent pegs, and the only
pegs that have one endpoint rather than two endpoints are pegs 1 and 
$n+2$.  An example of a web diagram in $W(4)$ is illustrated in
Figure~\ref{figurecasetwo}.

\begin{figure}[h!]
  \includegraphics{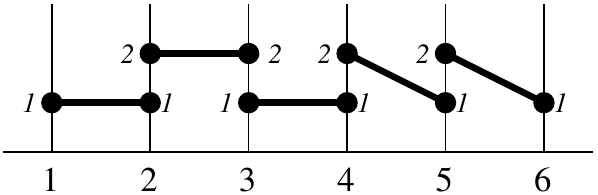}
  \caption{\label{figurecasetwo}An example of a web diagram of the
    form given in Equation~(\ref{casetwotype}).}
\end{figure}

Every diagram $D \in W(n)$ can be uniquely encoded as a sequence $\pi
\in (\pm 1)^n$ where
\begin{align*}
  \pi(i) &= \left\{   \begin{array}{cl}
      +1 & \mbox{ if } y_{i+1} = 1 \mbox{ and } x_{i+1}=2 \\
      -1 & \mbox{ if } y_{i+1} = 2 \mbox{ and } x_{i+1}=1 .
    \end{array} \right.
\end{align*}
  See Figure~\ref{piencode} for an illustration of this.  The example
  in Figure~\ref{figurecasetwo} is encoded by $\pi=(+1,-1,+1,+1)$.  We
  denote by $D_{\pi}$ the diagram that is encoded by the sequence
  $\pi$.

\begin{figure}[h!]
  \centerline{ \def\oheight{0.4} \def\xd{3}
	\includegraphics{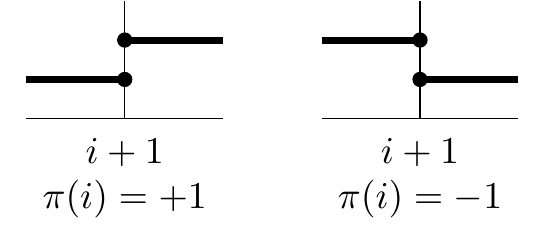}
}
\caption{\label{piencode}The value of $\pi(i)$ depends on how the
  edges meet at peg $i+1$.}
\end{figure}

A colouring of $D_{\pi}$ is a sequence $c \in \Colours{n+1}{k}$ where
$c(i)$ is the colour of edge $e_i$.  Given $c \in \Colours{n+1}{k}$
let $\DEA(c)=(\Des(c),\Equ(c),\Asc(c))$ where
\begin{align*}
  \Des(c) &=\{1\leq i \leq n: c(i)>c({i+1})\},\\
  \Equ(c) &= \{1\leq i \leq n : c(i)=c({i+1})\}, \mbox{ and}\\
  \Asc(c) &= \{1\leq i \leq n: c(i)<c({i+1})\}.
\end{align*}
We refer to $\Des(c)$, $\Equ(c)$ and $\Asc(c)$ as the {\it{descent}},
{\it{plateau}}, and {\it{ascent sets}}, respectively.  Let
$\Weuler_{n,k}(X_1,X_2,X_3) = \{ c \in \Colours{n+1}{k} :
\DEA(c)=(X_1,X_2,X_3) \}$, the set of colourings of $\Colours{n}{k}$
whose descent, plateau and ascent sets are given by $X_1$, $X_2$ and
$X_3$, respectively.  Set $\weuler_{n,k}(X_1,X_2,X_3) =
|\Weuler_{n,k}(X_1,X_2,X_3)|$.

Given diagrams $D_{\pi},D_{\sigma} \in W(n)$, let
$$Y(\pi,\sigma) = ( Y_1(\pi,\sigma), Y_2(\pi,\sigma), Y_3(\pi,\sigma), Y_4(\pi,\sigma))$$
be an ordered partition of $\{1,\ldots,n\}$ where
\begin{align*}
  Y_{1}(\pi,\sigma) &= \{ 1\leq i \leq n ~:~ \pi_i=+1 \mbox{ and } \sigma_i = -1 \} \\
  Y_{2}(\pi,\sigma) &= \{ 1\leq i \leq n ~:~ \pi_i=-1 \mbox{ and } \sigma_i = -1 \} \\
  Y_{3}(\pi,\sigma) &= \{ 1\leq i \leq n ~:~ \pi_i=+1 \mbox{ and } \sigma_i = +1 \} \\
  Y_{4}(\pi,\sigma) &= \{ 1\leq i \leq n ~:~ \pi_i=-1 \mbox{ and } \sigma_i = +1 \}.
\end{align*}
(We can summarise these more compactly as $i \in Y_{(5+2\sigma_i -
  \pi_i)/2} (\pi,\sigma)$ for all $1\leq i \leq n$.)

\begin{theorem}\label{casetwocount}
  Let $k \in \mathbb{N}$. Let $D_{\pi},D_{\sigma} \in W(n)$ and
  suppose that $Y(\pi,\sigma) = (Y_1,Y_2,Y_3,Y_4)$.  Then
  \begin{align*}
  F(D_{\pi},D_{\sigma},k) &= \bigcup_{A \subseteq Y_2 \atop B \subseteq Y_3} \Weuler_{n,k}(Y_1 \cup A, Y_2\cup Y_3 - (A \cup B), Y_4 \cup B),\\
  f(D_{\pi},D_{\sigma},k) &= \sum_{A \subseteq Y_2 \atop B \subseteq Y_3} \weuler_{n,k}(Y_1 \cup A, Y_2\cup Y_3 - (A \cup B), Y_4 \cup B).
\end{align*}
\end{theorem}

\begin{proof}
  Consider how changes occur at peg $i+1$ with respect to a colouring
  $c\in \Colours{n+1}{k}$.  If $\pi_i=+1$ and $\sigma_i = -1$, i.e., $i
  \in Y_1$, then the colour $c(i)$ must be greater than the colour
  $c({i+1})$ in order for the edges incident with peg $i+1$ to change
  their relative orders.
  So $c(i)> c({i+1})$ which means $i \in \Des(c)$ and we must have $Y_1 \subseteq \Des(c)$. This is illustrated in the following diagram:\\[1em]
  \centerline{ \def\oheight{0.4} \def\hxd{4} \def\xd{6}
	\includegraphics{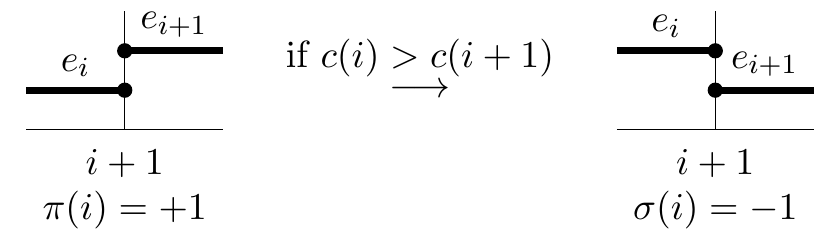}
}

If $\pi_i=-1$ and $\sigma_i = -1$ (i.e. $i \in Y_2$) then the colour
$c(i)$ must be greater than or equal to colour $c({i+1})$ in order for
the edges incident with peg $i+1$ to remain as they are.
So $c(i)\geq  c({i+1})$ which means $i \in \Des(c)\cup \Equ(c)$ and we must have $Y_2 \subseteq \Des(c)\cup \Equ(c)$. This is illustrated in the following diagram:\\[1em]
\centerline{ \def\oheight{0.4} \def\hxd{4} \def\xd{6}
\includegraphics{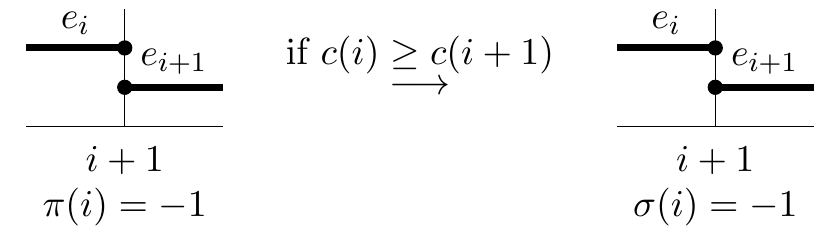}
} If $\pi_i=+1$ and $\sigma_i = +1$ (i.e. $i \in Y_3$) then the colour
$c(i)$ must be less than or equal to colour $c({i+1})$ in order for
the edges incident with peg $i+1$ to remain as they are.
So $c(i) \leq  c({i+1})$ which means $i \in \Asc(c)\cup \Equ(c)$ and we must have $Y_3 \subseteq \Asc(c)\cup \Equ(c)$. This is illustrated in the following diagram:\\[1em]
\centerline{ \def\oheight{0.4} \def\hxd{4} \def\xd{6}
\includegraphics{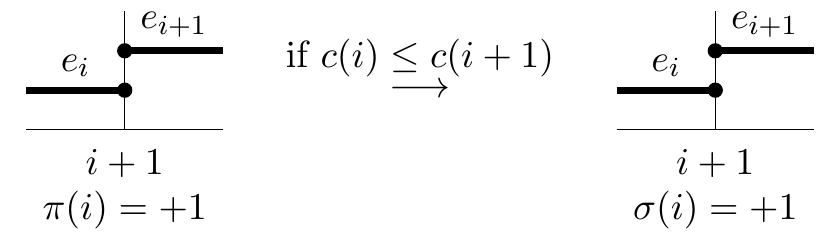}
} If $\pi_i=-1$ and $\sigma_i = +1$ (i.e. $i \in Y_4$) then the colour
$c(i)$ must be less than colour $c({i+1})$ in order for the edges
incident with peg $i+1$ to change their relative orders.
So $c(i) <  c({i+1})$ which means $i \in \Asc(c)$ and we must have $Y_4 \subseteq \Asc(c)$. This is illustrated in the following diagram:\\[1em]
\centerline{ \def\oheight{0.4} \def\hxd{4} \def\xd{6}
\includegraphics{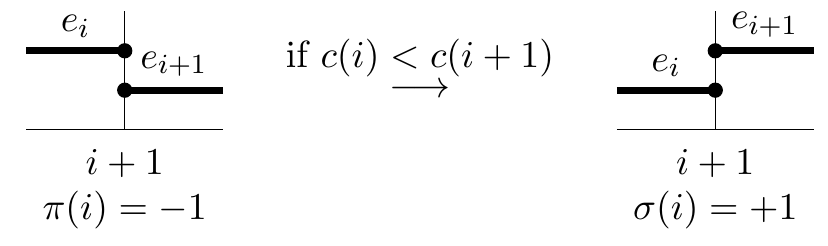}
} Combining these, we see that a colouring $c\in \Colours{n+1}{k}$
transforms $D_{\pi}$ to $D_{\sigma}$ iff $Y_1(\pi,\sigma) \subseteq
\Des(c)$, $Y_2(\pi,\sigma) \subseteq \Des(c) \cup \Equ(c)$,
$Y_3(\pi,\sigma) \subseteq \Asc(c) \cup \Equ(c)$, and $Y_4(\pi,\sigma)
\subseteq \Asc(c)$.  These four conditions are equivalent to $\Des(c)
= Y_1(\pi,\sigma) \cup A$, $\Asc(c) = Y_4(\pi,\sigma) \cup B$, and
$\Equ(c) = Y_2(\pi,\sigma) \cup Y_3(\pi,\sigma) -A-B$ for some sets
$A$ and $B$.
\end{proof}

The quantities $\MMmatrix{(W(n))}{(x)}{D_{\pi},D_{\sigma}} $ and
$\RRmatrix{(W(n))}{D_{\pi},D_{\sigma}}$ can now be calculated using
the expression in Theorem~\ref{casetwocount}.

Notice that for $D_{\pi} \in W(n)$ we have $Y(\pi,\pi)=(\emptyset,
Y_2, \{1,\ldots,n\}-Y_2,\emptyset)$.

\begin{corollary} \label{bermuda} Let $D_{\pi} \in W(n)$ with
  $Y(\pi,\pi)=(\emptyset, Y_2, \{1,\ldots,n\}-Y_2,\emptyset)$. Then
$$\RRmatrix{(W(n))}{D_{\pi},D_{\pi}} = \sum_k \dfrac{(-1)^{k-1}}{k} \sum_{A \subseteq Y_2 \atop B \subseteq \{1,\ldots,n\}-Y_2} \weuler_{n,k} (A,\{1,\ldots, n\}-A-B,B).$$
\end{corollary}

The following lemma is needed for the proofs of the two theorems that follow it.
\begin{lemma} \label{helpfulstirling} $\displaystyle\dfrac{1}{k!}
  \sum_{i=0}^k (-1)^{k-i} {k \choose i} (i+1)^n =
  \stirling{n+1}{k+1}$.
\end{lemma}

\begin{proof}
  Expand $(i+1)^n$ on the left-hand side and swap the order of
  summation to get
  \begin{align*}
    \dfrac{1}{k!} \sum_{i=0}^k (-1)^{k-i} {k \choose i} (i+1)^n 
  ~=~  
  \sum_{j=0}^n {n \choose j}  \dfrac{1}{k!} \sum_{i=0}^k (-1)^{k-i} {k \choose i} i^j 
  ~=~ \sum_{j=0}^n {n \choose j} \stirling{j}{k},
\end{align*}
where the last identity follows from a well known expression for the
Stirling numbers of the second kind (see ~\cite[Equation 1.94a]{ec1}).
The last sum here is the number of ways to choose a partition of a
$j$-element subset of $\{1,\ldots,n\}$ into $k$ blocks, for any $j$.
This is the same as the number of ways to partition the set
$\{1,\ldots,n+1\}$ into $k+1$ blocks, as the block which contains
$n+1$ represents those elements that were not chosen to be in any of
the $k$ blocks.  That number is therefore $\stirling{n+1}{k+1}$.
\end{proof}

\begin{theorem} \label{ctrt} For all $n\geq 1$, $\trace(\RRmatrix{(W(n))}{}) = 1$.
\end{theorem}

\begin{proof} Using Corollary~\ref{bermuda} we have:
  \begin{align*}
    \trace(\RRmatrix{(W(n))}{}) 
    &= \sum_{\pi \in (\pm 1)^n} \RRmatrix{(W(n))}{D_{\pi},D_{\pi}} \\
    &= \sum_{\pi \in (\pm 1)^n} \sum_{k=1}^{n+1} \dfrac{(-1)^{k-1}}{k} \sum_{A \subseteq Y_2 = Y_2(\pi,\pi) \atop B \subseteq [n]-Y_2} \weuler_{n,k} (A,[n]-A-B,B)\\
    &= \sum_{Y_2\subseteq [n]} \sum_{k=1}^{n+1} \dfrac{(-1)^{k-1}}{k} \sum_{A \subseteq Y_2 \atop B \subseteq [n]-Y_2} \weuler_{n,k} (A,[n]-A-B,B)\\
    &=  \sum_{k=1}^{n+1} \dfrac{(-1)^{k-1}}{k}  \sum_{A \subseteq [n]\atop B \subseteq [n]-A} \sum_{Y_2: A \subseteq Y_2 \subseteq [n]-B} \weuler_{n,k} (A,[n]-A-B,B)\\
    &=  \sum_{k=1}^{n+1} \dfrac{(-1)^{k-1}}{k}  \sum_{A \subseteq [n]\atop B \subseteq [n]-A} \weuler_{n,k} (A,[n]-A-B,B) 2^{n-|A|-|B|}.
  \end{align*}
The inner sum is 
\begin{align*}
  {\sum_{A \subseteq [n]\atop B \subseteq [n]-A} \weuler_{n,k} (A,[n]-A-B,B) 2^{n-|A|-|B|}} ~
  &= \sum_{c=\langle w,A \rangle \in \Colours{n+1}{k}} 2^{\equ(c)}\\
  &=\sum_{m=0}^n {n \choose n-m} 2^{n-m} |\Keys{m+1}{k}{}| \\
  &= \sum_{m=0}^n {n \choose n-m} 2^{n-m}  \sum_{i=0}^k (-1)^{k-i} {k \choose i} i (i-1)^m \\
  &=  \sum_{i=0}^k (-1)^{k-i}  {k \choose i} i  (i+1)^n.
\end{align*}
Thus
\begin{align*}
  \trace(\RRmatrix{(W(n))}{})
  &= \sum_{k=1}^{n+1} \dfrac{(-1)^{k-1}}{k} \sum_{v=1}^k (-1)^{k-v} v {k \choose v} (v+1)^n \\ 
  &= -\sum_{v=1}^{n+1} (-1)^v  (v+1)^n \sum_{k=v}^{n+1} { k-1 \choose v-1} \\
  &= -\sum_{v=1}^{n+1} (-1)^v  (v+1)^n {n+1\choose v}\\
  &= 1, 
\end{align*}
by using Lemma~\ref{helpfulstirling} with $k=n+1$.
\end{proof}

\begin{theorem} For all $n\geq 1$,
  $\displaystyle\trace(\MMmatrix{(W(n))}{(x)}{}) = \sum_{k=1}^{n+1}
  x^k k! \left( \stirling{n+2}{k+1} - \stirling{n+1}{k+1} \right)$.
\end{theorem}

\begin{proof}
  As in the proof of Theorem~\ref{ctrt},
  \begin{align*}
  \trace(\MMmatrix{(W(n))}{(x)}{}) ~
  &=~ \sum_{\pi \in (\pm 1)^n} \MMmatrix{(W(n))}{(x)}{D_{\pi},D_{\pi}} \\
  &=~  \sum_{k=1}^{n+1} x^k  \sum_{A \subseteq [n]\atop B \subseteq [n]-A} \weuler_{n,k} (A,[n]-A-B,B) 2^{n-|A|-|B|}.
\end{align*}
The inner sum may be written
\begin{align*}
  \lefteqn{\sum_{A \subseteq [n]\atop B \subseteq [n]-A} \weuler_{n,k} (A,[n]-A-B,B) 2^{n-|A|-|B|} }\\
  &= \sum_{i=0}^k (-1)^{k-i}  {k \choose i} i  (i+1)^n \\
  &= \sum_{i=0}^k (-1)^{k-i}  {k \choose i}  (i+1)^{n+1}  - \sum_{i=0}^k (-1)^{k-i}  {k \choose i}  (i+1)^{n}  \\
  &= k! \left( \stirling{n+2}{k+1} - \stirling{n+1}{k+1} \right),
\end{align*}
where the last identity follows from Lemma~\ref{helpfulstirling}. Thus
\begin{align*}
  \trace(\MMmatrix{(W(n))}{(x)}{}) ~&=~ \sum_{k=1}^{n+1} x^k  k! \left( \stirling{n+2}{k+1} - \stirling{n+1}{k+1} \right). \qedhere
\end{align*}
\end{proof}

\section{Case 3: A web world with $\Pegs(D)=(2,2,\ldots,2)$} \label{casethree}
In this section we consider a variant of the case in the previous
section.  This variant consists of inserting one more edge from peg 1
to peg $n$.  We will encode configurations as we did in the previous
section.  Fix $n \in \mathbb{N}$ and let $W(n)$ be the collection of
all web diagrams
\begin{align}
  D ~&=~ \{ e_i = (i,i+1,x_i,y_{i+1}) ~:~ 1 \leq i <n \} \cup \{ e_n=(1,n,y_1,x_n)\} \label{casethreetype}
\end{align}
where $\{x_i,y_i\}=\{1,2\}$ for all $1\leq i \leq n$.  So $W(n)$ is
the collection of web diagrams on $n$ pegs, consisting of $n$ edges
whose endpoints are on adjacent pegs, and with an edge between peg 1
and peg $n$.  Every such diagram can be encoded by a sequence $\pi =
(\pi(1),\ldots \ \pi(n)) \in (\pm 1)^n$ where $\pi(i) = x_i - y_i$ for
all $1\leq i \leq n$.

\begin{figure}[h!]
  \includegraphics{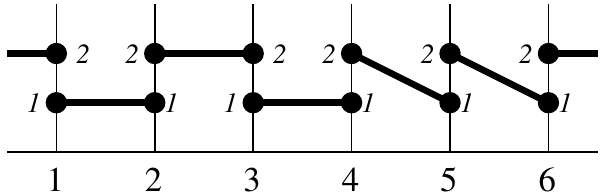}
  \caption{\label{figurecasethree}An example of a web diagram in
    $W(6)$ which is of the form given in Equation~(\ref{casethreetype}).
    The edge $e_6= (1,6,2,2)$ from peg $1$ to peg $6$ is illustrated
    in two parts on either extremities of the diagram.  This
    diagram is encoded by the sequence $\pi = (-1,+1,-1,+1,+1,+1)$.}
\end{figure}

Given $c \in \Colours{n}{k}$ and $D_{\pi} \in W(n)$, let $\Recon
{D_{\pi}}{c}$ be the web diagram that results by colouring edge $e_i$
with colour $c(i)$, for all $i$, and reconstructing.  Given $c \in \Colours{n}{k}$,
let $\DEAS(c) = (\Dess(c),\Equs(c),\Ascs(c))$ where 
\begin{align*}
  \Dess(c) &=\{1\leq i \leq n: c(i)>c({i+1}) \mbox{ with }c(n+1):=c(1)\},\\
  \Equs(c) &= \{1\leq i \leq n : c(i)=c({i+1})\mbox{ with }c(n+1):=c(1)\}, \mbox{ and}\\
  \Ascs(c) &= \{1\leq i \leq n: c(i)<c({i+1})\mbox{ with }c(n+1):=c(1)\}.
\end{align*}
Let $\dess(c)$, $\equs(c)$ and $\ascs(c)$ be the cardinalities of
the sets $\Dess(c)$, $\Equs(c)$ and $\Ascs(c)$, respectively.  Let
$\Weulers_{n,k}(X_1,X_2,X_3) = \{ c \in \Colours{n}{k} :
\DEAS(c)=(X_1,X_2,X_3) \}$, the set of colourings of $\Colours{n}{k}$
whose descent, plateau and ascents sets are given by $X_1$, $X_2$ and
$X_3$, respectively.  Set $\weulers_{n,k}(X_1,X_2,X_3) =
|\Weulers_{n,k}(X_1,X_2,X_3)|$.

Given diagrams $D_{\pi},D_{\sigma} \in W(n)$, let $Y(\pi,\sigma) = (
Y_1(\pi,\sigma), Y_2(\pi,\sigma), Y_3(\pi,\sigma), Y_4(\pi,\sigma))$
be an ordered partition of $\{1,\ldots,n\}$ where the $Y$'s are as
defined in Case 2.
\begin{theorem}\label{casethreecount}
  Let $k \in \mathbb{N}$. Let $D_{\pi},D_{\sigma} \in W(n)$ and
  suppose that $Y(\pi,\sigma) = (Y_1,Y_2,Y_3,Y_4)$.  Then
  \begin{align*}
  F(D_{\pi},D_{\sigma},k) &= \bigcup_{A \subseteq Y_2 \atop B \subseteq Y_3} \Weulers_{n,k}(Y_1 \cup A, Y_2\cup Y_3 - (A \cup B), Y_4 \cup B)\\
  f(D_{\pi},D_{\sigma},k) &= \sum_{A \subseteq Y_2 \atop B \subseteq Y_3} \weulers_{n,k}(Y_1 \cup A, Y_2\cup Y_3 - (A \cup B), Y_4 \cup B).
\end{align*}
\end{theorem}

We omit the proof of Theorem~\ref{casethreecount} because it is almost
identical to the proof of Theorem~\ref{casetwocount}.  The only places
in which they differ is in taking account of the colouring of the
solitary edge between peg 1 and peg $n$.

\begin{theorem} \label{case3trace} 
For all $n\geq 1$,
$\trace(\RRmatrix{(W(n))}{}) = n+1$.
\end{theorem}

\begin{proof}
  $\begin{array}[t]{rcl}
\trace(\RRmatrix{(W(n))}{}) 
&=& \displaystyle \sum_{\pi \in (\pm 1)^n} \RRmatrix{(W(n))}{D_{\pi},D_{\pi}} \\
&=&\displaystyle  \sum_{\pi \in (\pm 1)^n} \sum_k \dfrac{(-1)^{k-1}}{k} 
	\sum_{A \subseteq Y_2 = Y_2(\pi,\pi)\atop B\subseteq [n]-Y_2} \weulers_{n,k} (A,[n]-A-B,B) 2^{n-|A|-|B|}.
      \end{array}$
      \ \\
      The inner sum is
      \begin{align*}
        \sum_{A \subseteq Y_2(\pi,\pi)\atop B\subseteq [n]-Y_2(\pi,\pi)} \weulers_{n,k} (A,[n]-A-B,B) 2^{n-|A|-|B|}
        &= \sum_{c \in \Colours{n}{k}} 2^{\equs(c)} \\
        &= \sum_{c \in \Colours{n}{k}\atop c(1)\neq c(n)} 2^{\equ(c)} + \sum_{c \in \Colours{n}{k}\atop c(1)= c(n)} 2^{1+\equ(c)}.
      \end{align*}
The first sum is 
\begin{align*}
  \sum_{c \in \Colours{n}{k}\atop c(1)\leq c(n)} 2^{\equ(c)} 
  &= \sum_{m=k}^n \sum_{w \in \Keys{m}{k}{\neq}} 2^{n-m} {n-1\choose m-1} \\
  &= \sum_{m=k}^n 2^{n-m} {n-1\choose m-1} | \Keys{m}{k}{\neq} |,
\end{align*}
and the second sum is
\begin{align*}
  \sum_{c \in \Colours{n}{k}\atop c(1)= c(n)} 2^{1+\equ(c)} 
  &= \sum_{m} \sum_{w \in \Keys{m}{k}{=}} 2^{1+n-m} {n-1\choose m-1} \\
  &= \sum_{m} 2^{1+n-m} {n-1\choose m-1} | \Keys{m}{k}{=} |.
\end{align*}
Adding the two equations above, and using the expressions given in Lemma~\ref{keyslemma}, we have
$$
\sum_{A \subseteq Y_2(\pi,\pi)\atop B\subseteq [n]-Y_2(\pi,\pi)}
\weulers_{n,k} (A,[n]-A-B,B) 2^{n-|A|-|B|} = \sum_{i=0}^k (-1)^{k-i}
{k \choose i} \left( (i+1)^n + (i-1)\right).$$ Replacing this into the
equation for the trace
\begin{align*}
  \trace(\RRmatrix{(W(n))}{}) 
  &= \sum_{k=1}^n \dfrac{(-1)^{k-1}}{k} \sum_{i=0}^k (-1)^{k-i} {k \choose i} \left( (i+1)^n + (i-1)\right)\\
  &= \sum_{i=1}^n \dfrac{(-1)^{i+1}}{i} {n \choose i} \left( (i+1)^n +i-1\right).
\end{align*}
Since $((i+1)^n+i-1)/i=1+\sum_{j=1}^n {n \choose j} i^{j-1}$,
\begin{align*}
  \trace(\RRmatrix{(W(n))}{}) 
  &= \sum_{i=1}^n (-1)^{i+1} {n \choose i} + \sum_{i,j=1}^n (-1)^{i+1} {n\choose i} {n \choose j} i^{j-1}\\
  &= 1+\sum_{j=1}^n {n \choose j} \sum_{i=1}^n (-1)^{i+1} {n \choose i} i^{j-1} \\
  &= 1+\sum_{j=1}^n {n \choose j} \left( -0^{j-1}(-1)^n + \sum_{i'=0}^n (-1)^{i'} {n \choose i'} (n-i')^{j-1}\right) \\
  &= 1+\sum_{j=1}^n {n \choose j} \left( -0^{j-1}(-1)^n + n! \stirling{j-1}{n}\right).
\end{align*}
As $j-1<n$, the Stirling number $\stirling{j-1}{n}$ is 0. The other
term in the parentheses is also zero except when $j=1$.  Thus
$$\trace(\RRmatrix{(W(n))}{}) 
= 1+{n\choose 1} (-1)^{n+1} (-1(-1)^n) = 1+n.\qedhere
$$
\end{proof}

\begin{theorem} For all $n \geq 1$,
  $\displaystyle\trace(\MMmatrix{(W(n))}{(x)}{}) = x+ \sum_{k=1}^{n+1}
  x^k k! \stirling{n+1}{k+1}$.
\end{theorem}

\begin{proof}
  As in the proof of the previous theorem,
  \begin{align*}
  \trace(\MMmatrix{(W(n))}{(x)}{}) 
  &= \sum_{\pi \in (\pm 1)^n} \MMmatrix{(W(n))}{(x)}{D_{\pi},D_{\pi}} \\
  &= \sum_{k=1}^{n+1} x^k  \sum_{A \subseteq [n]\atop B \subseteq [n]-A} \weulers_{n,k} (A,[n]-A-B,B) 2^{n-|A|-|B|}.\\
  &= \sum_{k=1}^{n+1} x^k \sum_{i=0}^k (-1)^{k-i} {k \choose i} \left( (i+1)^n + (i-1)\right)\\
  &= x+ \sum_{k=1}^{n+1} x^k k! \stirling{n+1}{k+1},
\end{align*}
by Lemma~\ref{helpfulstirling}.
\end{proof}

\section{Transitive web worlds}
In this section we consider a special class of web worlds that we term
{\it{transitive web worlds}}.  The defining property of these web worlds
is the simple condition (in terms of the web diagram) that every
peg (other than the first and last) is the right endpoint of at least
one edge and the left endpoint of at least one edge.

\begin{definition}
  Let $W$ be a web world and $D\in W$ with $\PegSet(D)=\{1,\ldots,n\}$.
  We will call $W$ {\it{transitive}} if each of the following hold;
  \begin{enumerate}
  \item[(i)] $(1,k) \in \PegpairsSet(D)$ for some $1<k\leq n$.
  \item[(ii)] $(k,n) \in \PegpairsSet(D)$ for some $1\leq k< n$.
  \item[(iii)] For all $k$ with $1<k<n$ there exists $z$ and $z'$ such that $1\leq z<k<z'\leq n$
	such that $(z,k)$ and $(k,z')$ are both in $\PegpairsSet(D)$.
  \end{enumerate}
\end{definition}

The condition for $W$ to be transitive is equivalent to $\Represent(W)$
being such that the only rows and columns of all zeros are the
leftmost column and the bottom row.

\begin{example} Consider all web worlds with exactly 3 edges. There
  are 30 of these. Only 5 are transitive, represented by the following
  matrices:

$$
\left(\begin{matrix}
    0 & 3 \\
    0 & 0
  \end{matrix}\right),
\qquad \left(\begin{matrix}
    0 & 2 & 0 \\
    0 & 0 & 1 \\
    0 & 0 & 0
  \end{matrix}\right),
\qquad \left(\begin{matrix}
    0 & 1 & 1 \\
    0 & 0 & 1 \\
    0 & 0 & 0
  \end{matrix}\right),
\qquad \left(\begin{matrix}
    0 & 1 & 0 \\
    0 & 0 & 2 \\
    0 & 0 & 0
  \end{matrix}\right),
\qquad \left(\begin{matrix}
    0 & 1 & 0 & 0\\
    0 & 0 & 1 & 0 \\
    0 & 0 & 0 & 1 \\
    0 & 0 & 0 & 0
  \end{matrix}\right).
$$
\end{example}

The matrices which represent transitive web worlds are in one-to-one
correspondence with the upper triangular matrices introduced in
~\cite{robert} (to go from the former to the latter remove the
leftmost column and the bottom row).  A striking property of the
latter class of matrices is their rich connection to several other
combinatorial objects, namely (2+2)-free partially ordered sets,
pattern avoiding permutations, Stoimenow matchings, and sequences of
integers called ascent sequences~\cite{tpt,composition,djk}.

Cases 2 and 3, in Sections \ref{case2} and \ref{casethree}, are
examples of transitive web worlds.  Results concerning the above
structures are readily applicable to transitive web worlds.  For example,
the generating function for the number of transitive web worlds according
to the number of edges, number of pegs, and number of distinct pairs
of pegs connected by edges, is given in ~\cite[Theorem 11]{remmel}.

\section{Further comments and open problems}
Results presented in this paper will be explained in a physics context in the forthcoming paper ~\cite{DGMSW}.
There are many questions outstanding in understanding the properties of these new structures. 
We list here a small collection of these questions, each of whose answers would represent a significant step forward in this new area.
Let $W$ be a web world and let $D$ and $E$ be web diagrams in $W$.

\begin{enumerate}
\item[(1)] In Theorem~\ref{thm33}, in order to apply this theorem we must
have that the web diagram $D$, when written as a sum of indecomposable web diagrams
$D=E_1 \oplus \ldots \oplus E_k$, is such that all of its constituent indecomposable web diagrams 
$(E_1,\ldots,E_k)$ are different. How can this theorem be extended to the case for when there are
repeated constituent web diagrams. The simplest such to analyse is the 
web diagram $D=\{(1,2,1,1),(1,2,2,2)\} = \{(1,2,1,1)\} \oplus \{(1,2,1,1)\}$.
\item[(2)] Can the off-diagonal terms $\MMmatrix{(W)}{(x)}{D,E}$ and
  $\RRmatrix{(W)}{D,E}$ be calculated in terms of mappings between the
  posets $P(D)$ and $P(E)$?
\item[(3)] Find a direct way to perform the calculations in
  Equations~(\ref{eastbound}) and~(\ref{westbound}) using $\Represent(W)$.
\item[(4)] Can the web worlds $W$ with $\trace(\RRmatrix{(W)}{})=1$ be
  characterized?
\item[(5)] Can one find other exactly solvable examples of web worlds
  than those presented here? By exactly solvable we mean that it is
  possible to give an expression for every entry of the matrices
  $\MMmatrix{(W)}{(x)}{}$ and $\RRmatrix{(W)}{}$.
\item[(6)] Can the relationship between the representing matrices
  $\Represent(W)$ for transitive web diagrams and $(2+2)$-free posets (or
  other combinatorial objects in bijection with these) be exploited in
  order to simplify calculations related to the trace?
\item[(7)] 
From the physics perspective (see ~\cite{WIM})  the main question is the determination 
of the left eigenvectors of the mixing matrix $\RRmatrix{}{}$.
The first step in this direction is to determine how many such 
eigenvectors have eigenvalue 1, which is the number of independent 
contributions to the exponent of the Eikonal amplitude from the given web world.
It is given by the rank of $\RRmatrix{}{}$ (which equals the trace of $\RRmatrix{}{}$).
This task was achieved here for certain classes of webs. 
The natural question is what is the general formula for 
this invariant given the entries of $\Represent (W)$.
\end{enumerate}


\end{document}